\documentclass[11pt,a4paper]{article}

\usepackage{titlesec}
\usepackage{fancyhdr}
\usepackage{a4wide}
\usepackage{graphicx}
\usepackage{float}
\usepackage{amssymb}
\usepackage{amsmath}
\usepackage{amsfonts}
\usepackage{amsthm}
\usepackage{color}
\usepackage{mathrsfs}
\usepackage{array}
\usepackage{eucal}
\usepackage{tikz}
\usepackage[T1]{fontenc}
\usepackage{inputenc}
\usepackage[english]{babel}
\usepackage{lmodern}
\usepackage{hyperref}
\usepackage{geometry}
\usepackage{changepage}
\usepackage{bm}
\changepage{0pt}{}{}{}{}{0pt}{}{0pt}{6pt}
\usepackage[numbers]{natbib}
\usepackage{hyperref}
\hypersetup{
pdfpagemode=none,
pdftoolbar=true,        
pdfmenubar=true,        
pdffitwindow=false,     
pdfstartview={Fit},    
pdftitle={Acoustic passive cloaking using thin outer resonators},    
pdfauthor={L. Chesnel, J. Heleine, S.A. Nazarov},     
pdfsubject={},  
pdfcreator={L. Chesnel, J. Heleine, S.A. Nazarov},   
pdfproducer={L. Chesnel, J. Heleine, S.A. Nazarov}, 
pdfkeywords={}, 
pdfnewwindow=true,      
colorlinks=true,       
linkcolor=magenta,          
citecolor=red,        
filecolor=cyan,      
urlcolor=blue           
}

\newcommand{\dsp}{\displaystyle}

\newcommand{\eps}{\varepsilon}
\newcommand{\om}{\omega}
\newcommand{\Om}{\Omega}
\newcommand{\mrm}[1]{\mathrm{#1}}

\newcommand{\Cplx}{\mathbb{C}}
\newcommand{\N}{\mathbb{N}}
\newcommand{\R}{\mathbb{R}}
\newcommand{\Z}{\mathbb{Z}}

\newcommand{\mL}{\mrm{L}}

\newtheorem{theorem}{Theorem}[section]
\newtheorem{lemma}[theorem]{Lemma}
\newtheorem{remark}[theorem]{Remark}

\newtheorem{proposition}[theorem]{Proposition}
\newtheorem{procedure}[theorem]{Procedure}

\begin{document}

~\vspace{0.0cm}
\begin{center}
{\sc \bf\huge Acoustic passive cloaking using \\[6pt]thin outer resonators}
\end{center}

\begin{center}
\textsc{Lucas Chesnel}$^1$, \textsc{J\'er\'emy Heleine}$^1$, \textsc{Sergei A. Nazarov}$^{2}$\\[16pt]
\begin{minipage}{0.95\textwidth}
{\small
$^1$ INRIA/Centre de math\'ematiques appliqu\'ees, \'Ecole Polytechnique, Institut Polytechnique de Paris, Route de Saclay, 91128 Palaiseau, France;\\
$^2$ Institute of Problems of Mechanical Engineering, Russian Academy of Sciences, V.O., Bolshoj pr., 61, St. Petersburg, 199178, Russia;\\[2pt]

E-mails: \texttt{lucas.chesnel@inria.fr}, \texttt{jeremy.heleine@inria.fr}, \texttt{srgnazarov@yahoo.co.uk} \\[-14pt]
\begin{center}
(\today)
\end{center}
}
\end{minipage}
\end{center}
\vspace{0.4cm}

\noindent\textbf{Abstract.} 
We consider the propagation of acoustic waves in a 2D waveguide unbounded in one direction and containing a compact obstacle. The wavenumber is fixed so that only one mode can propagate. The goal of this work is to propose a method to cloak the obstacle. More precisely, we add to the geometry thin outer resonators of width $\eps$ and we explain how to choose their positions as well as their lengths to get a transmission coefficient approximately equal to one as if there were no obstacle. In the process we also investigate several related problems. In particular, we explain how to get zero transmission and how to design phase shifters. The approach is based on asymptotic analysis in presence of thin resonators. An essential point is that we work around resonance lengths of the resonators. This allows us to obtain effects of order one with geometrical perturbations of width $\eps$. Various numerical experiments illustrate the theory.\\

\noindent\textbf{Key words.} Acoustic waveguide, passive cloaking, asymptotic analysis, thin resonator, scattering coefficients, complex resonance.

\section{Introduction}

In this article, we propose a method to cloak, in some sense that we describe hereafter, an obstacle embedded in an acoustic waveguide. The problem of cloaking an object has a large number of applications and has been the subject of intense studies over the last decade in the theory of waves propagation. In particular, the methods based on anomalous resonances 
\cite{MiNi06,BoSc10,KLSW14,Nguy17,MaSh16} and on transformation optics \cite{PeSS06,Leon06,Norr08,Nguy10,ChCh10,GKLU09b} have encountered a large success which is related to the development of metamaterials. In the present study, we consider a rather academic but universal problem of wave propagation in a waveguide which is unbounded in one direction. This problem also arises in electromagnetism and in water-wave theory in certain configurations. We work in time-harmonic regime and the wavenumber is set so that one mode can propagate in the structure. The scattering of this mode by the obstacle generates a reflection and a transmission characterised by some complex reflection and transmission coefficients (see (\ref{Field1}), (\ref{Field2}) for precise definitions). The goal of this work is to explain how to perturb the initial geometry to obtain a new waveguide where the reflection coefficient is approximately zero and the transmission coefficient is approximately equal to one as in the reference strip without obstacle.\\ 
\newline
We emphasize that our objective is less ambitious than what people try to do when working with transformation optics. In particular, we do not aim at having the same field as if there were no obstacle outside of some cloaking  device. Our goal is simply to get invisibility at infinity with one single propagating mode. The interesting counterpart is that we will be able to achieve it without using metamaterials with unconventional properties which are still hard to produce. Besides, we think that this setting is relevant in many situations because in waveguides, the evanescent part of the field that we neglect is exponentially decaying at infinity and therefore is really difficult to distinguish from noise a few wavelengths far from the obstacle. Actually we do not use particular material but instead play with the geometry of the waveguide. Let us mention also that we do not add active sources in the system as people do in active cloaking \cite{Mill06,VaMO09,VaMO11,NoAP12,OSMMM15,Chee16,CDGG21}. No, what we realize is passive cloaking at infinity by perturbing the shape of the waveguide.\\
\newline
The main difficulty of the problem lies in the fact that the dependence of the scattering coefficients with respect to the geometry is not explicit and not linear. In order to address it, techniques of optimization have been considered. We refer the reader in particular to \cite{AlLS17,LDOHG19,LGHDO19,Lebb19}. However the functionals involved in the analysis are non convex and unsatisfying local minima exist. Moreover, these methods do not allow the user to control the main features of the shape compare to the approach we propose below.\\
\newline
To cloak obstacles, we will work with thin outer resonators of width $\eps$ as depicted for example in Figures \ref{FigCloakElephant}--\ref{FigCloakPenetrable}. We already used similar techniques in \cite{ChNa20Distrib} and \cite{ChHN21} to design respectively energy distributors and mode converters. Thus the present article constitutes somehow the third opus of the trilogy started with \cite{ChNa20Distrib,ChHN21}. In general, due to the geometrical properties of waveguides as in Figures \ref{FigCloakElephant}--\ref{FigCloakPenetrable}, the thin resonators produce only a perturbation of order $\eps$ on the fields and so on the scattering coefficients. However working around the resonance lengths of the resonators (see (\ref{ConditionResonance})), we can get effects of order one. This will be a key property in our approach. Note that this has been studied for example in \cite{Krieg,BoTr10,LiZh17,LiZh18,LiSZ19} in a context close to ours, namely in the study of the scattering of an incident wave by a periodic array of subwavelength slits. The core of our approach is based on an asymptotic expansion of the scattering solutions with respect to $\eps$ as $\eps$ tends to zero. This will allow us to derive formula for the scattering coefficients with a relatively explicit dependence on the geometrical features. To obtain the expansions, we will apply techniques of matched asymptotic expansions. For related methods, we refer the reader to \cite{Beal73,Gady93,KoMM94,Naza96,Gady05,Naza05,JoTo06,BaNa15,BoCN18}. We emphasize that an important feature of our study distinguishing it from the previous references is that the lengths, and not only the widths, of the resonators depend on $\eps$ (see (\ref{DefL})). This way of considering the problem, which was proposed in \cite{NaCh21a,NaCh21b,ChNa20Distrib,ChHN21}, is an essential ingredient to reveal the resonant phenomena. From this perspective, our work shares similarities with \cite{HoSc19,BrHS20,BrSc20} (see also references therein). On the other hand, we will observe in the asymptotic procedure that when $\eps$ tends to zero, everything happens like if punctual sources were located at the junction points between the resonators and the main part of the waveguide (see e.g. (\ref{DefGamma})). Therefore our approach somehow has some connections with works on active cloaking \cite{Mill06,VaMO09,VaMO11,NoAP12,OSMMM15,Chee16,CDGG21}.\\
\newline
The outline is as follows. We describe the setting and the notation in Section \ref{SectionSetting}. Then we explain how to achieve almost zero reflection by working with one thin outer resonator of width $\eps$ in Section \ref{SectionZeroR}. This section is rather long. First, we describe in detail the computation of an asymptotic expansion of the scattering solution with respect to $\eps$. Then we use it to obtain almost zero reflection. Finally, we discuss the particular case of symmetric geometries  and make some digressions concerning zero transmission. In Section \ref{SectionPhaseShift}, we show how to construct phase shifters, that is geometries where the reflection coefficient is approximately zero and the phase of the transmission coefficient can take any value on the unit circle. To proceed, we work in the reference strip with two thin resonators.
In Section \ref{SectionCloaking}, first we gather the results of the two previous sections to explain how to cloak obstacles with three resonators. Then we show that cloaking can be achieved with only two resonators. In Section \ref{SectionConclusion}, we discuss possible extensions and related open questions. Finally, we establish some auxiliary results needed in the analysis in a short appendix. The main results of this work are Procedure \ref{MainProce} (almost zero reflection), Procedure \ref{ProcePhaseShift2} (phase shifter) and Procedure \ref{ProceCloakingTwo} (cloaking with two resonators).

\newpage 
\section{Setting}\label{SectionSetting}

\begin{figure}[!ht]
\centering
\begin{tikzpicture}[scale=2]
\draw[fill=gray!30] (0.3,1) circle (0.4) ;
\draw[fill=gray!30,draw=none](-1.5,0) rectangle (1.5,1);
\draw (-1.5,0)--(1.5,0);
\draw (-1.5,1)--(-0.1,1);
\draw (0.7,1)--(1.5,1);
\draw[fill=white] (-0.5,0.2)--(0.1,0.8)--(0.4,0.3)--cycle;
\draw[dashed] (-1.5,0)--(-2,0);
\draw[dashed] (-1.5,1)--(-2,1);
\draw[dashed] (1.5,0)--(2,0);
\draw[dashed] (1.5,1)--(2,1);
\draw (-0.8,-0.05)--(-0.8,0.05);
\draw (1.3,-0.05)--(1.3,0.05);
\node at (-0.8,-0.2){\small $-d$};
\node at (1.3,-0.2){\small $d$};
\begin{scope}[shift={(-3,0)}]
\draw[->] (0,0.2)--(0.6,0.2);
\draw[->] (0.1,0.1)--(0.1,0.7);
\node at (0.6,0.3){\small $x$};
\node at (0.2,0.7){\small $y$};
\end{scope}
\node at (-1.3,0.1){\small $\Om$};
\end{tikzpicture}
\caption{Examples of geometry with obstacle. \label{DomainOriginal}} 
\end{figure}
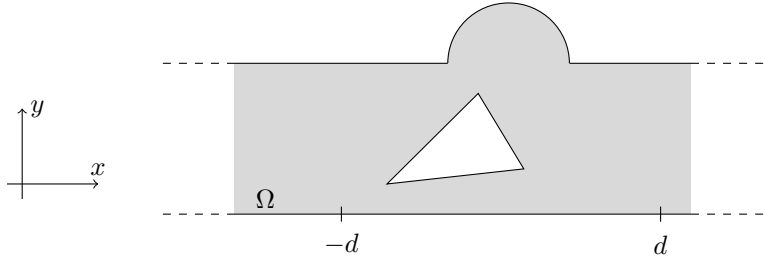

Let $\Om$ be a domain which coincides with the reference strip $S:=\{z=(x,y)\in\R\times(0;1)\}$ outside of a bounded region, say $[-d;d]^2$ for a certain $d>0$ (see Figure \ref{DomainOriginal}). We assume that $\Om$ is connected and that its boundary $\partial\Om$ is Lipschitz. In the sequel we shall often refer to the region where $\Om$ differs from $S$ as an ``obstacle'' though perturbations of the walls as depicted in Figure \ref{DomainOriginal} can be considered as well. Interpreting the domain $\Om$ as an acoustic waveguide, we are led to study the following problem with Neumann Boundary Conditions (BC)
\begin{equation}\label{MainPb}
 \begin{array}{|rcll}
 \Delta u +\omega^2  u&=&0&\mbox{ in }\Omega\\
 \partial_\nu u &=&0 &\mbox{ on }\partial\Omega.
\end{array}
\end{equation}
Here, $\Delta$ is the Laplace operator while $\partial_\nu$ corresponds to the derivative along the exterior normal. Furthermore, $u$ is the acoustic pressure of the medium while $\omega>0$ is the wave number. We fix $\om\in(0;\pi)$ so that only the  modes
\[
\mrm{w}^\pm(x,y)=e^{\pm i\om x}
\]
can propagate. We are interested in the solutions to the diffraction problem \eqref{MainPb} generated by the incoming waves $\mrm{w}^{\pm}$ coming from $\mp\infty$. These scattering solutions admit the decompositions
\begin{equation}\label{Field1}
W^+(x,y)=\begin{array}{|ll}
\mrm{w}^+(x,y)+R_+\mrm{w}^-(x,y)+\dots &\quad\mbox{ for }x<-d\\[3pt]
\phantom{\mrm{w}^+(x,y)+\ \,} T\,\mrm{w}^+(x,y)+\dots &\quad\mbox{ for }x>d\end{array}
\end{equation}
\begin{equation}\label{Field2}
W^-(x,y)=\begin{array}{|ll}
\phantom{\mrm{w}^-(x,y)+\ \,} T\,\mrm{w}^-(x,y)+\dots &\quad\mbox{ for }x<-d\\[3pt]
\mrm{w}^-(x,y)+R_-\mrm{w}^+(x,y)+\dots &\quad\mbox{ for }x>d\end{array}
\end{equation}
where $R_{\pm}\in\mathbb{C}$ are reflection coefficients and $T\in\mathbb{C}$ is a transmission coefficient. Here the ellipsis stand for remainders which decay at infinity with the rate $e^{-(4\pi^2-\om^2)^{1/2}|x|}$. To set ideas, if trapped modes\footnote{We remind the reader that trapped modes are non zero solutions of (\ref{MainPb}) which are in $\mL^2(\Om)$ and therefore which decay exponentially at infinity.} exist for the problem (\ref{MainPb}), which happens only in rare circumstances, $W^{\pm}$ are defined as the functions which are orthogonal to these trapped modes for the inner product of $\mL^2(\Om)$. From $R_{\pm}$, $T$, we form the scattering matrix
\[
\mathbb{S}:=\left(\begin{array}{cc}
R_+ & T \\[2pt]
T & R_-
\end{array}\right)\in\Cplx^{2\times2}.
\]
The matrix $\mathbb{S}$ is symmetric not necessarily hermitian and it is known that it is also unitary: $\mathbb{S}\overline{\mathbb{S}}^{\top}=\mrm{Id}_{2\times 2}$. For a proof of these results with similar notation, we refer the reader for example to \cite{na489} or \cite[Proposition 3.2]{BeBC21}. Here the unitary of $\mathbb{S}$ implies that we have the relations
\begin{equation}\label{RelConsNRJ}
|R_{\pm}|^2+|T|^2=1,\qquad\qquad \overline{R_+}T+\overline{T}R_-=0.
\end{equation}
Above, the identities on the left traduce the conservation of energy. In the analysis below, we will make use of a result which says that the structure of the scattering matrix translates into a structure for the scattering solutions (see \cite[Theorem 2]{na489} or \cite[Proposition 3.4]{BeBC21}). 
\begin{proposition}\label{propositionStructure}
Set $W:=(W^+,W^-)^{\top}$ where $W^{\pm}$ are the scattering solutions introduced in (\ref{Field1}), (\ref{Field2}). Then in $\Om$ we have the identities
\begin{equation}\label{IdentityStructure}
\mathbb{S}\overline{W}=W\qquad\Leftrightarrow\qquad
\left\{\begin{array}{rcl}
R_+\overline{W^+}+T\overline{W^-}&=&W^+ \\[2pt]
T\overline{W^+}+R_-\overline{W^-}&=&W^-.
\end{array}\right.
\end{equation}
\end{proposition}
\noindent The main goal of this work is to explain how to perturb slightly $\Om$ to get new geometries where the reflection coefficients are approximately equal to zero and the transmission coefficient is approximately equal to one as in the unperturbed strip $S=\R\times(0;1)$. In this case, the initial obstacles in $\Om$ are approximately cloaked, and more precisely made invisible to an observer measuring scattered fields at some distance of them and unable to detect the influence of evanescent components (which is always the case in practice due to the presence of noise). 

\section{Zero reflection with one thin outer resonator}\label{SectionZeroR}

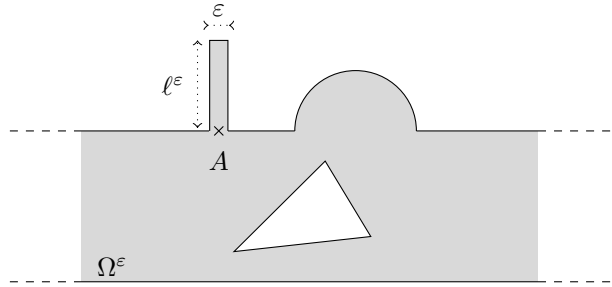
\begin{figure}[!ht]
\centering
\begin{tikzpicture}[scale=2]
\draw[fill=gray!30] (0.3,1) circle (0.4) ;
\draw[fill=gray!30,draw=none](-1.5,0) rectangle (1.5,1);
\draw (-1.5,0)--(1.5,0);
\draw (-1.5,1)--(-0.1,1);
\draw (0.7,1)--(1.5,1);
\draw[dotted,<->] (-0.74,1.02)--(-0.74,1.6);
\draw[dotted,>-<] (-0.7,1.7)--(-0.5,1.7);
\node at (-0.6,1.8){\small $\eps$};
\node at (-0.9,1.3){\small $\ell^\eps$};
\node at (-0.6,0.8){\small $A$};
\draw[fill=gray!30,draw=none](-0.66,0.9) rectangle (-0.54,1.6);
\draw (-0.66,1)--(-0.66,1.6)--(-0.54,1.6)-- (-0.54,1);
\draw[fill=white] (-0.5,0.2)--(0.1,0.8)--(0.4,0.3)--cycle;
\draw[dashed] (-1.5,0)--(-2,0);
\draw[dashed] (-1.5,1)--(-2,1);
\draw[dashed] (1.5,0)--(2,0);
\draw[dashed] (1.5,1)--(2,1);
\draw (-0.63,0.97)--(-0.57,1.03);
\draw (-0.63,1.03)--(-0.57,0.97);
\node at (-1.3,0.1){\small $\Om^{\eps}$};
\end{tikzpicture}
\caption{Geometry with one thin outer resonator. \label{DomainWithOneLigament}} 
\end{figure}

The first step in our approach to achieve invisibility consists in obtaining zero reflection working with one thin outer resonator.
\subsection{Description of the geometry}
For $\eps>0$ small, set 
\begin{equation}\label{DefL}
\ell^{\eps}:=\ell^0+\eps \ell'
\end{equation}
where the values of $\ell^0>0$ and $\ell'\in\R$ will be fixed later to observe interesting phenomena. Then define the thin resonator
\begin{equation}\label{DefOneLigament}
L^{\eps}:=(p-\eps/2;p+\eps/2)\times[1;1+\ell^{\eps})
\end{equation}
for some $p\in\R$. We assume that $\Om$ and $L^{\eps}$ are such that $\overline{\Om}\cap L^{\eps}=(p-\eps/2;p+\eps/2)\times\{1\}$ and we set $A:=(p,1)\subset\partial\Om$ (the junction point, see an illustration with Figure \ref{DomainWithOneLigament}). Note that the length and not only the width of the resonator $L^{\eps}$ depends on $\eps$. We will make more precise this dependence later. Finally, set 
\[
\Om^{\eps}:=\Om\cup L^{\eps}.
\]
As for $\Om$, we assume that the parameters are such that $\Om^{\eps}$ is connected and that its boundary $\partial\Om^{\eps}$ is Lipschitz. Otherwise consider another position for the resonator (let us mention that we could also place the resonator on the inferior wall, that is to take $L^{\eps}$ of the form $L^{\eps}=(p-\eps/2;p+\eps/2)\times(-\ell^{\eps};0]$). Changing the definition of $d>0$ if necessary, we can also assume that $\Om^{\eps}$ coincides with $S=\R\times(0;1)$ outside of $[-d;d]^2$. We consider the problem
\begin{equation}\label{MainPbPerturbed}
 \begin{array}{|rcll}
 \Delta u^{\eps} +\omega^2  u^{\eps}&=&0&\mbox{ in }\Omega^{\eps}\\
 \partial_\nu u^{\eps} &=&0 &\mbox{ on }\partial\Omega^{\eps}.
\end{array}
\end{equation}
The scattering of the wave $\mrm{w}^+$ coming from $-\infty$ leads us to study the solution of (\ref{MainPbPerturbed}) admitting the decomposition
\begin{equation}\label{decompoMainTermPerturbed}
u^{\eps}(x,y)=\begin{array}{|ll}
\mrm{w}^+(x,y)+R^{\eps}_+\mrm{w}^-(x,y)+\dots &\quad\mbox{ for }x<-d\\[3pt]
\phantom{\mrm{w}^+(x,y)+\ \,} T^{\eps}\,\mrm{w}^+(x,y)+\dots &\quad\mbox{ for }x>d.\end{array}
\end{equation}
Here $R^{\eps}_+\in\Cplx$ is a reflection coefficient, $T^{\eps}\in\Cplx$ is a transmission coefficient and again the ellipsis stand for remainders which decay at infinity with the rate $e^{-(4\pi^2-\om^2)^{1/2}|x|}$. We have the relation of conservation of energy
\[
|R_+^{\eps}|^2+|T^{\eps}|^2=1.
\]
In general, the thin resonator has almost no influence on the scattering of the wave $\mrm{w}^+$ and therefore on the scattering coefficients. More precisely, in general one can establish that as $\eps$ tends to zero, there holds 
\[
 R^\eps_+  = R_++o(1),\qquad\qquad T^\eps=T+o(1),
\]
where $R_+$, $T$ are the scattering coefficients in the unperturbed geometry introduced in (\ref{Field1}). The main goal of this section is to show that by choosing carefully the lengths of the thin resonator as well as its position, we can get almost zero reflection. More precisely, we will establish that by choosing carefully the properties of the resonator, as $\eps$ tends to zero, we can have
\[
 R^\eps_+  = o(1),\qquad\qquad T^\eps=T^0+o(1),
\]
with $|T^0|=1$. This does not yield perfect cloaking because in general $T^0\ne1$ but this is a first step.\\
\newline
In the sequel, we compute an asymptotic expansion of the function $u^\eps$ appearing in (\ref{decompoMainTermPerturbed}) as $\eps$ tends to zero. This will give us an expansion of the coefficients $R^{\eps}_+$, $T^{\eps}$. To proceed, first we introduce some auxiliary objects which will be useful in the analysis.

\subsection{Auxiliary objects} 

$\star$ Considering the limit $\eps\rightarrow0^+$ in the equation \eqref{MainPbPerturbed} restricted to the thin resonator $L^{\eps}$ of length $\ell^\eps=\ell^0+\eps\ell'$, we are led to study the one-dimensional Helmholtz problem with mixed BC
\begin{equation}\label{Pb1D}
\begin{array}{|ll}
\partial^2_yv+\omega^2v=0\qquad\mbox{ in }(1;1+\ell^0)\\[3pt]
v(1)=\partial_yv(1+\ell^0)=0.
\end{array}
\end{equation}
Note that the condition $v(1)=0$ is imposed artificially. Eigenvalues and eigenfunctions (up to a multiplicative constant) of the problem (\ref{Pb1D}) are given by
\[
(\pi (m+1/2)/\ell^0)^2,\qquad\quad v(y)=\sin(\pi(m+1/2)(y-1)/\ell^0)\qquad\quad\mbox{ with }m\in\N:=\{0,1,2,3,\dots\}.
\]
Importantly, in the sequel we shall choose $\ell^0$ such that at the limit $\eps\to0$, the domain $L^{\eps}$ is resonant for the problem (\ref{Pb1D}). In other words, we select $\ell^0$ such that 
\begin{equation}\label{ConditionResonance}
\om\,\ell^0=\pi(m+1/2) 
\end{equation}
for some $m\in\N$. Let us emphasize that the limit problem (\ref{Pb1D}) is set in the fixed segment $(1;1+\ell^0)$. But the true lengths $\ell^\eps=\ell^0+\eps \ell'$ of the resonator $L^{\eps}$ depends on the parameter $\eps$. This is an essential element in the analysis to bring to light the resonant phenomena.\\

\begin{figure}[!ht]
\centering
\begin{tikzpicture}[scale=1.3]
\begin{scope}
\clip (-1.8,2) -- (1.8,2) -- (1.8,0) -- (-1.8,0) -- cycle;
\draw[fill=gray!30,draw=none](0,2) circle (1.8);
\end{scope}
\draw[fill=gray!30,draw=none](-0.5,2) rectangle (0.5,3.5);
\draw (-1.8,2)--(-0.5,2)--(-0.5,3.5);
\draw (1.8,2)--(0.5,2)--(0.5,3.5);
\draw[dashed] (-0.5,3.5)--(-0.5,4);
\draw[dashed] (0.5,3.5)--(0.5,4);
\draw[dashed] (2.1,2)--(1.8,2);
\draw[dashed] (-2.1,2)--(-1.8,2);
\draw[thick,draw=gray!30](-0.495,2)--(0.495,2);
\node[mark size=1pt,color=black] at (0,2) {\pgfuseplotmark{*}};
\node[color=black] at (0,1.8) {\small $O$};
\begin{scope}[shift={(1.9,1)}]
\draw[->] (0,0)--(0.5,0);
\draw[->] (0,0)--(0,0.5);
\node at (0.66,0){\small$\xi_x$};
\node at (0,0.66){\small$\xi_y$};
\end{scope}
\node at (0,1){\small$\Xi^-$};
\node at (0,2.75){\small$\Xi^+$};
\node at (-1.5,2.2){\small$\Xi$}; 
\draw[dotted] (-0.5,2)--(0.5,2);
\end{tikzpicture}
\caption{Geometry of the inner field domain $\Xi$.\label{FrozenGeom}} 
\end{figure}
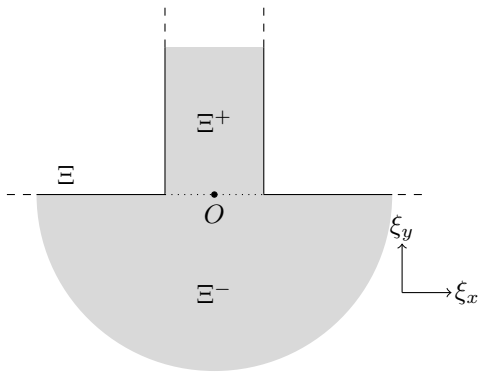

\noindent $\star$ Now we present a second problem which is involved in the construction of asymptotics and which will be used to describe the boundary layer phenomenon near the junction point $A$. To capture rapid variations of the field in the vicinity of $A$, we introduce the stretched coordinates $\xi=(\xi_x,\xi_y)=\eps^{-1}(z-A)=(\eps^{-1}(x-p),\,\eps^{-1}(y-1))$. Observing that 
\begin{equation}\label{Strechted1}
(\Delta_z+\om^2)u^{\eps}(\eps^{-1}(z-A))=\eps^{-2}\Delta_{\xi}u^{\eps}(\xi)+\dots,
\end{equation}
we see that when $\eps$ tends to zero, the main term involves simply a Laplace operator and we are led to consider the Neumann problem 
\begin{equation}\label{PbBoundaryLayer}
-\Delta_\xi Y=0\qquad\mbox{ in }\Xi,\quad\qquad \partial_\nu Y=0\quad\mbox{ on }\partial\Xi.
\end{equation}
Here $\Xi:=\Xi^-\cup\Xi^+\subset\mathbb{R}^2$ (see Figure \ref{FrozenGeom}) is the union of the half-plane $\Xi^-$ and the semi-strip $\Xi^+$ such that
\[
\Xi^-:={\mathbb R}^2_-=\{\xi=(\xi_x,\xi_y):\,\xi_y<0\},\qquad\qquad
\Xi^+:=\{\xi:\,\xi_y\geq 0, |\xi_x|<1/2\}.
\]
In the method of matched asymptotic expansions (see
the monographs \cite{VD,Ilin}, \cite[Chpt. 2]{MaNaPl} and others) that we will use, we will work with solutions of \eqref{PbBoundaryLayer} which are bounded or which have polynomial growth in the semi-strip as $\xi_y\rightarrow+\infty$ as well as logarithmic growth in the half plane as $|\xi|\rightarrow+\infty$. One of such solutions is evident and is given by $Y^0=1$.
Another solution, which is linearly independent with $Y^0$, is the unique function satisfying (\ref{PbBoundaryLayer}) and which has the representation
\begin{equation}\label{PolyGrowth}
Y^1(\xi)=\left\{
\begin{array}{ll}
\xi_y+C_\Xi+O(e^{-\pi \xi_y})& \mbox{ as }\xi_y\rightarrow+\infty,\quad \xi
\in \Xi^+\\[5pt]
\dsp\frac{1}{\pi}\ln \frac{1}{|\xi|}+O
 \Big(\frac{1}{|\xi|}\Big)& \mbox{ as }|\xi|\rightarrow+\infty,\quad \xi
\in \Xi^-.
\end{array}\right.
\end{equation}
Here, $C_\Xi$ is a universal constant whose value can be computed using conformal mapping, see for example \cite{Schn17}. Note that the coefficients in front of the growing terms in \eqref{PolyGrowth} are related due to the fact that a harmonic function has zero total flux at infinity. For the existence of $Y^1$ and the uniqueness of its definition, we refer the reader for example to \cite[Lemma 4.1]{BoCN18} (for general formally self-adjoint boundary value problems, one can look at \cite[\S5]{Naza99b}).

\subsection{Asymptotic analysis}\label{paragraphAsympto}

In this section, we compute an asymptotic expansion of the field $u^{\eps}$ appearing in (\ref{MainPbPerturbed}) as $\eps$ tends to zero. The final results are summarized in Proposition \ref{PropoAsymOneLig}.\\ 
\newline
In the waveguide $\Om$ (without the resonator), we work with the ansatz
\begin{equation}\label{AnsatzWaveguide}
u^\eps=u^0+\eps u' +\dots \quad\mbox{\rm in }\Om,
\end{equation}
while in the thin resonator, we consider the expansion 
\[
u^\eps(x,y)=\eps^{-1}v^{-1}(y)+v^0(y)+\dots\quad\mbox{\rm in }L^\eps.
\]
Here the ellipsis stand for higher order terms which are not important in our analysis. Taking the formal limit $\eps\to0^+$, we find that $v^{-1}$ must solve the homogeneous problem (\ref{Pb1D}). Note in particular that the condition $v^{-1}(1)=0$ comes from the fact that the expansion (\ref{AnsatzWaveguide}) of $u^{\eps}$ in $\Om$ remains bounded as $\eps$ tends to zero. Under the assumption (\ref{ConditionResonance}) for the length $\ell^0$, we must take $v^{-1}$ of the form 
\[
v^{-1}(y)=a {\bf v}(y)\qquad\mbox{ with }\qquad a\in\mathbb{C},\ {\bf v}(y)=\sin(\omega (y-1)).
\]  
Let us stress that the value of $a$ is unknown and will be fixed during the construction of the asymptotics of $u^\eps$. 
At the point $A$, the Taylor formula gives
\begin{equation}\label{DefInnerA}
\eps^{-1}v^{-1}(y)+v^0(y)=0+(C^{A} \xi_y  +v^0(1))+\dots\qquad\mbox{ with }\qquad C^{A}:=a\partial_y {\bf v}(1)=a\omega.
\end{equation}
Here $\xi_y=\eps^{-1}(y-1)$ is the stretched variable introduced just before (\ref{Strechted1}).\\
\newline
We look for an inner expansion of $u^{\eps}$ in the vicinity of $A$ of the form
\[
u^\eps(x)=C^{A}\,Y^1(\xi)+ c^{A}+\dots
\]
where $Y^1$ is introduced in (\ref{PolyGrowth}), $C^{A}$ is defined in (\ref{DefInnerA}) and $c^{A}$ is a constant to determine.\\
\newline 
Let us continue the matching procedure. Taking the limit $\eps\to0^+$, we find that the main term $u^0$ in (\ref{AnsatzWaveguide}) must solve the problem
\[
\Delta u^0 +\omega^2u^0=0\ \mbox{ in }\Om,\qquad 
\partial_\nu u^0=0\mbox{ on }\partial\Om\setminus \{A\},
\]
with the expansion 
\[
u^{0}(x,y)=\begin{array}{|ll}
\mrm{w}^+(x,y)+R^{0}_+\mrm{w}^-(x,y)+\dots &\quad\mbox{ for }x<-d\\[3pt]
\phantom{\mrm{w}^+(x,y)+\ \,} T^{0}\,\mrm{w}^+(x,y)+\dots &\quad\mbox{ for }x>d.\end{array}
\]
Here $R^{0}_+,\,T^0\in\Cplx$ and again the ellipsis stand for remainders which decay exponentially at infinity. The coefficients $R^{0}_+,\,T^0$ will provide the first terms in the asymptotics of $R^{\eps}_+,\,T^\eps$:
\[
R^{\eps}_+=R^{0}_++\dots\qquad\mbox{ and }\qquad T^\eps=T^0+\dots\,.
\]
Matching the behaviours of the inner and outer expansions of $u^{\eps}$ in $\Om$, we find that at the point $A$, the function $u^0$ must expand as 
\[
u^0(x,y)= C^{A}\frac{1}{\pi}\ln \frac{1}{r^{A}}+U^0+O(r^{A})
\qquad \mbox{ when }r^{A}:=((x-p)^2+(y-1)^2)^{1/2}\rightarrow0^+,
\]
where $U^0$ is a constant. Observe that $u^0$ is singular at $A$. Integrating by parts in 
\begin{equation}\label{IntegRho1}
0=\int_{\Om^\kappa}(\Delta u^0 +\omega^2u^0)W^{\pm}-u^0\,(\Delta W^{\pm} +\omega^2 W^{\pm})\,dz,
\end{equation}
with $\Om^\kappa:=\{(x,y)\in\Om\,,|x|<\kappa\mbox{ and }r^A>1/\kappa\}$, and taking the limit $\kappa\to+\infty$, we get 
\[
\begin{array}{|l}
2i\om(R^{0}_+-R_+)+C^{A}W^{+}(A)=0\\[4pt]
2i\om(T^{0}-T)+C^{A}W^{-}(A)=0.
\end{array}
\]
From the expression of $C^{A}$ (see (\ref{DefInnerA})), this gives
\begin{equation}\label{equation1}
\begin{array}{|l}
R^{0}_+=R_++iaW^{+}(A)/2\\[4pt]
T^{0}=T+iaW^{-}(A)/2.
\end{array}
\end{equation}
Then matching the constant behaviour between the outer expansion and the inner  expansion inside $\Om$, we get
\[
U^0=C^{A}\,\pi^{-1}\ln\eps+c^{A}=-C^{A}\,\pi^{-1}|\ln\eps|+c^{A}.
\]
This sets the value of $c^{A}$. However $U^0$ depends on $a$ and we have to explicit this dependence. For $u^0$, we can consider the decomposition
\begin{equation}\label{SecondDecompo}
u^0=W^++C^{A}\gamma
\end{equation}
where $\gamma$ is the outgoing function such that
\begin{equation}\label{DefGamma}
\begin{array}{|rcll}
\Delta \gamma+\om^2\gamma&=&0&\mbox{ in }\Om\\
\partial_\nu\gamma&=&\delta_{A}&\mbox{ on }\partial\Om.
\end{array}
\end{equation}
Here $\delta_{A}$ stands for the Dirac delta function at $A$. Denote by $\Gamma$ the constant behaviour of $\gamma$ at $A$, that is the constant such that $\gamma$ behaves as 
\begin{equation}\label{DefGammaM}
\gamma(x,y)= \frac{1}{\pi}\ln \frac{1}{r^{A}}+\Gamma+O(r^{A})\qquad \mbox{ when }r^{A}\rightarrow0^+.
\end{equation}
Then from (\ref{SecondDecompo}), we derive
\[
U^0=W^+(A)+a\om \Gamma.
\]
Matching the constant behaviour at $A$ inside the thin resonator $L^{\eps}$, we obtain
\begin{equation}\label{BoundaryCondition1}
\begin{array}{lcl}
v^0(1)&=&C^{A}\,C_{\Xi}+c^{A} = U^0+C^{A}\,(\pi^{-1}|\ln\eps|+C_{\Xi})\\[5pt]
 &=& W^+(A)+a\om\,(\pi^{-1}|\ln\eps|+C_{\Xi}+\Gamma).
\end{array}
\end{equation}
Writing the compatibility condition so that the problem (\ref{Pb1D}), supplemented with the condition (\ref{BoundaryCondition1}) instead of $v^0(1)=0$, admits a solution, we get
\begin{equation}\label{CompatibilityCondition}
v^0\partial_y\textbf{v}|_{y=1}-v^0\partial_y\textbf{v}|_{y=1+\ell^0}-(\textbf{v}\partial_y v^0|_{y=1}-\textbf{v}\partial_y v^0|_{y=1+\ell^0})=0.
\end{equation}
Since $\textbf{v}(1)=\partial_y\textbf{v}(1+\ell^0)=0$, we obtain
\[
\om v^0(1)+(-1)^{m}\partial_yv^0(1+\ell^0)=0.
\]
On the other hand, from $\partial_\nu(\eps^{-1}a {\bf v})(1+\ell^\eps)+v^0(\ell^\eps)+\dots =0$, we infer that $\partial_y v^0(1+\ell^0)=\om^2 a\ell'\sin(\om \ell^0)=(-1)^{m}\om^2 a\ell'$. Thus we get
\begin{equation}\label{Part2System}
W^+(A)+a\om\,(\pi^{-1}|\ln\eps|+C_{\Xi}+\Gamma+\ell')=0.
\end{equation}
Below, see Lemmas \ref{LemmaCReal} and \ref{lemmaRelConstants}, we prove that $C_{\Xi}\in\R$ and $\Im m\,(\om\Gamma)=(|W^+(A)|^2+|W^-(A)|^2)/4$. Thus we have
\[
a(\eta+i(|W^+(A)|^2+|W^-(A)|^2)/4)=-W^+(A)
\]
with 
\[
\eta:=\om(\pi^{-1}|\ln\eps|+C_{\Xi}+\Re e\,\Gamma+\ell'). 
\]
Gathering (\ref{equation1}) and (\ref{Part2System}), we obtain the system
\[
\begin{array}{|l}
R^{0}_+=R_++iaW^{+}(A)/2\\[4pt]
T^{0}=T+iaW^{-}(A)/2\\[4pt]
a(\eta+i(|W^+(A)|^2+|W^-(A)|^2)/4)=-W^+(A).
\end{array}
\]
Solving this system, we get
\begin{equation}\label{SolutionSystem}
\begin{array}{c}
\dsp R^{0}_+=R_+-\cfrac{2i(W^{+}(A))^2}{4\eta+i(|W^+(A)|^2+|W^-(A)|^2)}\ ,\quad T^{0}=T-\cfrac{2iW^{+}(A)W^{-}(A)}{4\eta+i(|W^+(A)|^2+|W^-(A)|^2)}\ ,\\[12pt]
a=\cfrac{-4W^+(A)}{(4\eta+i(|W^+(A)|^2+|W^-(A)|^2))}\ .
\end{array}
\end{equation}
This ends the asymptotic analysis of $u^{\eps}$, $R^{\eps}_+$, $T^{\eps}$ as $\eps$ tends to zero. Let us summarize these results. 
\begin{proposition}\label{PropoAsymOneLig}
Assume that
\begin{equation}\label{DefRegLong}
\ell^{\eps}= \pi(m+1/2)/\om+\eps(\eta/\om-C_{\Xi}-\Re e\,\Gamma-\pi^{-1}|\ln\eps|)
\end{equation}
for some $m\in\N$ and $\eta\in\R$. Then when $\eps$ tends to zero, we have the following expansions
\begin{equation}\label{AsymptoFinalResults1}
\fbox{$\begin{array}{l}
u^{\eps}(x,y)=W^+(x,y)+a\om\gamma(x,y)+o(1) \quad\mbox{ in }\Om,\\[6pt]
u^{\eps}(x,y)=\eps^{-1}a\sin(\om (y-1))+O(1) \quad\mbox{ in }L^\eps,\\[6pt]
R_{+}^{\eps}=R_{+}^{0}+o(1),\qquad T^{\eps}=T^0+o(1), \\[6pt]
\mbox{where $a$, $R_{+}^{0}$, $T^0$ are given by (\ref{SolutionSystem}).}
\end{array}$}
\end{equation}
Here $\gamma$ is the function introduced in (\ref{DefGamma}). 
\end{proposition}
\noindent Let us make a few comments concerning this result. First, despite of the presence of the term $|\ln\eps|$ in the definition of the length $\ell^{\eps}$ in (\ref{DefRegLong}), note that $\ell^{\eps}$ in (\ref{DefRegLong}) converges to $\pi(m+1/2)/\om$  when $\eps$ tends to zero. But the important message here is that by choosing the way $\ell^{\eps}$ converges to $\pi(m+1/2)/\om$, in particular by changing the parameter $\eta\in\R$ in (\ref{DefRegLong}), we obtain different limits for the scattering coefficients $R_{+}^{\eps}$, $T^{\eps}$ (see an illustration with Figure \ref{Illustration}). As a consequence, the scattering coefficients considered as functions of both the width and the length of the resonator are not continuous at the point $(0,\pi(m+1/2)/\om)$. Second, we see from (\ref{AsymptoFinalResults1}) that when $a\ne0$, which is equivalent to have $W^+(A)\ne0$, the amplitude of the field $u^{\eps}$ in the resonator blows up as $\eps$ tends to zero. The case $a=0$, or equivalently $W^+(A)=0$, corresponds to a situation where, roughly speaking, the resonant eigenfunction associated with the complex resonance existing due to the presence of the thin resonator is not excited. Finally, note that a direct calculus using (\ref{RelConsNRJ}) and the relations (\ref{IdentityStructure}) shows that independently of the choice of $\eta$ in (\ref{SolutionSystem}), we have
\begin{equation}\label{RelConsNRJ}
|R_{+}^0|^2+|T^0|^2=1.
\end{equation}
This is coherent with the conservation of energy.

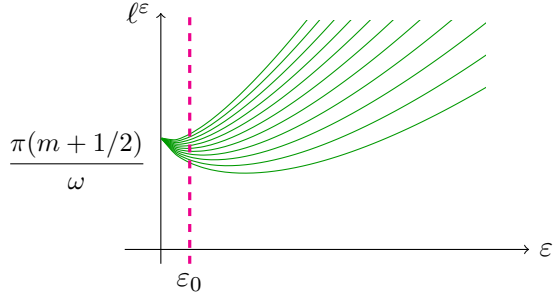
\begin{figure}[!ht]
\centering
\raisebox{-0.2cm}{\begin{tikzpicture}[scale=0.95]
\draw[->] (-0.5,0) -- (5.1,0) node[right] {$\eps$};
\draw[->] (0,-0.2) -- (0,3.3) node[left] {$\ell^\eps$};
\node at (-0.1,3.14/2-0.3) [left] {\small $\cfrac{\pi(m+1/2)}{\om}$};
\begin{scope}
\clip(-2,-0.5) rectangle (4.5,3.2);
\draw[domain=0.01:4.7,smooth,variable=\x,green!60!black] plot ({\x},{3.14/2+\x*(-0.5+log10(\x))});
\draw[domain=0.01:4.7,smooth,variable=\x,green!60!black] plot ({\x},{3.14/2+\x*(-0.4+log10(\x))});
\draw[domain=0.01:4.7,smooth,variable=\x,green!60!black] plot ({\x},{3.14/2+\x*(-0.3+log10(\x))});
\draw[domain=0.01:4.7,smooth,variable=\x,green!60!black] plot ({\x},{3.14/2+\x*(-0.2+log10(\x))});
\draw[domain=0.01:4.7,smooth,variable=\x,green!60!black] plot ({\x},{3.14/2+\x*(-0.1+log10(\x))});
\draw[domain=0.01:4.7,smooth,variable=\x,green!60!black] plot ({\x},{3.14/2+\x*(-0.0+log10(\x))});
\draw[domain=0.01:4.7,smooth,variable=\x,green!60!black] plot ({\x},{3.14/2+\x*(0.5+log10(\x))});
\draw[domain=0.01:4.7,smooth,variable=\x,green!60!black] plot ({\x},{3.14/2+\x*(0.4+log10(\x))});
\draw[domain=0.01:4.7,smooth,variable=\x,green!60!black] plot ({\x},{3.14/2+\x*(0.3+log10(\x))});
\draw[domain=0.01:4.7,smooth,variable=\x,green!60!black] plot ({\x},{3.14/2+\x*(0.2+log10(\x))});
\draw[domain=0.01:4.7,smooth,variable=\x,green!60!black] plot ({\x},{3.14/2+\x*(0.1+log10(\x))});
\end{scope}
\draw[magenta,dashed,very thick] (0.4,-0.2) -- (0.4,3.3);
\node at (0.4,-0.2) [below] {$\eps_0$};
\end{tikzpicture}}
\caption{Paths $\{(\eps,\pi(m+1/2)/\om+\eps(\eta/\om-C_{\Xi}-\Re e\,\Gamma-\pi^{-1}|\ln\eps|),\,\eps>0\}\subset\mathbb{R}^2$ for several values of $\eta$.  According to the chosen path, the limit of the scattering coefficients along this path as $\eps\to0^+$ is different. With this picture, we can also understand that for a fixed small $\eps_0$, the scattering coefficients have a rapid variation as the length of the resonator changes in a vicinity of $\pi(m+1/2)/\om$ (see similar explanations in \cite{NaCh21a,NaCh21b}).  \label{Illustration}}
\end{figure}

\subsection{Almost zero reflection}

Now we explain how to use the results above to get almost zero reflection. More precisely, we prove the following statement, the main result of the section.
\begin{proposition}\label{PropositionZeroRef}
Assume that $R_+\ne0$ and $T\ne0$. Then there are some positions of $A=(p,1)$ such that there holds $R^{0}_+(\eta)=0$ for some $\eta\in\R$. Here $R^{0}_+$ is the main term in the asymptotics of $R_{+}^{\eps}$, see (\ref{AsymptoFinalResults1}). For such $A$, Proposition \ref{PropoAsymOneLig} ensures that one can make $R_{+}^{\eps}$ as small as one wishes by taking $\eps$ small enough and by tuning correctly the length of the resonator $L^{\eps}$.
\end{proposition}
\begin{remark}\label{RmkIntermLig}
Note that we exclude the case $R_+=0$ because in this situation we already have zero reflection and there is no need for adding a resonator. In the case $T=0$, the proof below does not work. In this situation, one possibility to get zero reflection is to add first one or several resonators to obtain a transmission coefficient quite different from zero. And then to add another well-tuned resonator to kill the reflection. Let us mention that this strategy is also interesting when $T$ is small but non zero because in this case the Procedure \ref{MainProce} proposed below can be quite unstable.
\end{remark}
\begin{proof}
First, let us observe that for $\alpha\in\R\setminus\{0\}$, results on M\"obius transformation guarantee that the set 
\begin{equation}\label{Mobius}
\{\cfrac{1}{\eta+i\alpha}\,|\,\eta\in\R\}
\end{equation}
coincides with $\mathscr{C}(-i/(2\alpha),1/(2\alpha))\setminus\{0\}$ where $\mathscr{C}(-i/(2\alpha),1/(2\alpha))$ is the circle centred at $-i/(2\alpha)$ of radius $1/(2\alpha)$. We infer that when $W^+(A)W^-(A)\ne0$, for $\eta$ varying in $\R$, the coefficients $R_{+}^{0}(\eta)$, $T^{0}(\eta)$ in (\ref{SolutionSystem}) run on circles (see Figure \ref{FigFishMatSca} left).\\
\newline
On the other hand, we find that when $W^+(A)\ne0$, we have $R_{+}^{0}(\eta)=0$ if and only if 
\begin{equation}\label{sweepR}
4\eta+i(|W^+(A)|^2+|W^-(A)|^2)=2i(W^+(A))^2/R_+.
\end{equation}
We see that (\ref{sweepR}) is valid for a certain $\eta\in\R$ if and only if
\begin{equation}\label{WithoutApprox}
|W^+(A)|^2+|W^-(A)|^2=2\Re e\,((W^+(A))^2/R_+).
\end{equation}
The difficulty in the verification of this identity lies in the fact that in general we do not have explicit formula for $W^{\pm}$ which would allows us to assess $W^{\pm}(A)$. To bypass this problem, we will look at a situation where the resonator is located far from the obstacle. To set ideas, we assume that it is on the left. When $p$ tends to $-\infty$, due to decompositions (\ref{Field1})--(\ref{Field2}), we know that there is $p_0\in\R$ such that there exists $C>0$, independent of $p\le p_0$, such that
\begin{equation}\label{ErrorEstim}
|W^+(A)- (e^{i\om p}+R_+\,e^{-i\om p})|+|W^-(A)- T\,e^{-i\om p}| \le C\,e^{-\sqrt{4\pi^2-\om^2}\,|p|}.
\end{equation}
Since $T\ne0$, this guarantees that $W^+(A)W^-(A)\ne0$ for $p$ small enough. Define the functionals 
\[
\begin{array}{rcl}
\mathscr{F}(p)&:=&|W^+(A)|^2+|W^-(A)|^2-2\Re e\,((W^+(A))^2/R_+)\\[3pt]
\mathscr{F}^{\mrm{asy}}(p)&:=&|e^{i\om p}+R_+\,e^{-i\om p}|^2+|T\,e^{-i\om p}|^2-2\Re e\,((e^{i\om p}+R_+\,e^{-i\om p})^2/R_+)
\end{array}
\]
(the superscript $^{\mrm{asy}}$ stands for ``asymptotic''). Observe that (\ref{WithoutApprox}) is valid if and only $\mathscr{F}(p)=0$. 
From (\ref{ErrorEstim}), there holds 
\begin{equation}\label{ErrorEstimFunc}
|\mathscr{F}(p)-\mathscr{F}^{\mrm{asy}}(p)| \le C\,e^{-\sqrt{4\pi^2-\om^2}\,|p|}
\end{equation}
for $p\le p_0$. The constant $C>0$ may change from one line to another but remains independent of $p\le p_0$. A direct calculation yields 
\[
\begin{array}{rcl}
\mathscr{F}^{\mrm{asy}}(p)&=&1+|R_+|^2+|T|^2+2\Re e\,(R_+\,e^{-2i\om p})-2(\Re e\,(e^{2i\om p}/R_+)+2+\Re e\,(R_+e^{-2i\om p}))\\[3pt]
&=&2(1+\Re e\,(e^{2i\om p}/R_+))
\end{array}
\]
(the second equality is a consequence of (\ref{RelConsNRJ})). Thus there holds $\mathscr{F}^{\mrm{asy}}(p)=0$ if and only if $\Re e\,(e^{2i\om p}/R_+)=-1$. If we set $R_+=\rho e^{i\theta_+}$ with $\rho>0$ and $\theta_+\in[0;2\pi)$, this is equivalent to have
\begin{equation}\label{RelationToGetZeroRLeft}
\cos(2\om p-\theta_+)=-\rho.
\end{equation}
Since $\rho\in(0;1)$, this equation always has some solutions of the form 
$p_{\star}+n\pi/\om$ with $n\in\Z$ (observe actually that there are two families of solutions). Besides, we find
\begin{equation}\label{CalculusDerivative}
\partial_{p}\mathscr{F}^{\mrm{asy}}(p_{\star})=\Re e\,(2i\om\,e^{2i\om p}/R_+)=\Re e\,(2i\om\,e^{i(2\om p-\theta_+)})=\pm2\om\sqrt{1-\rho^2}\ne0.
\end{equation}
Relation (\ref{CalculusDerivative}) guarantees that $\mathscr{F}^{\mrm{asy}}$ changes sign in a neighbourhood of $p^{\star}$. If $s>0$ is a small given parameter, from estimate (\ref{ErrorEstimFunc}) and the intermediate values theorem, we deduce that $p\mapsto \mathscr{F}(p)$ vanishes on $[p_{\star}-n\pi/\om-s;p_{\star}-n\pi/\om+s]$ for $n\in\N$ large enough. As a consequence, relation (\ref{WithoutApprox}) holds at least for an almost periodic sequence of points $A=(p,1)$ with $p\to-\infty$. In this case, we have $R^{0}_+(\eta)=0$ for a certain $\eta\in\R$.
\end{proof}
\noindent Numerically, one can compute $W^{\pm}(A)$ and look for positions of $A$ such that (\ref{WithoutApprox}) is satisfied. Another possibility is to place the resonator quite far on the left of the obstacle at the abscissa $p$ satisfying (\ref{RelationToGetZeroRLeft}). If the resonator is a bit far on the right, we have 
$W^+(A)\approx T\,e^{i\om p}$ and $W^-(A)\approx e^{-i\om p}+R_-\,e^{i\om p}$ so that 
\begin{equation}\label{ApproximationIdentityBis}
\begin{array}{cl}
&|W^+(A)|^2+|W^-(A)|^2-2\Re e\,((W^+(A))^2/R_+) \\[4pt]
\approx & 1+|R_-|^2+|T|^2+2\Re e\,(R_-\,e^{2i\om p})-2\Re e\,(T^2e^{2i\om p}/R_+)\\[4pt]
\end{array}
\end{equation}
Using that $T/R_+=-\overline{T}/\overline{R_-}$ (see (\ref{RelConsNRJ})), we find that the right hand side of (\ref{ApproximationIdentityBis}) cancels if and only if $1+\Re e\,(R_-\,e^{2i\om p})+|T|^2\Re e\,(e^{2i\om p}/\overline{R_-})=0$. Writing $R_-=\rho e^{i\theta_-}$ with $\rho>0$ and $\theta_-\in[0;2\pi)$, one can check that this is equivalent to have
\begin{equation}\label{RelationToGetZeroRRight}
\begin{array}{cl}
\cos(2\om p+\theta_-)=-\rho.
\end{array}
\end{equation}

\noindent Finally we can state the following procedure to get almost zero reflection.
\begin{procedure}\label{MainProce}
Let $R_{\pm}=\rho e^{i\theta_{\pm}}$ with $\rho\in(0;1)$, $\theta_{\pm}\in[0;2\pi)$ be the reflection coefficients introduced in (\ref{Field1}). Place the resonator $L^{\eps}$ a bit far on the left (resp. on the right) of the obstacle at the position $A=(p,1)$ with $p$ satisfying (\ref{RelationToGetZeroRLeft}) (resp. (\ref{RelationToGetZeroRRight})) or (\ref{sweepR}). Then for any fixed $m\in\N$, for $\eps$ small enough, there is one length of the resonator, close to $\pi(m+1/2)/\om$, such that $R_+^{\eps}\approx0$.
\end{procedure}

\begin{remark}\label{RemarkBelowCritical}
For a given small $\eps>0$, let us denote by $\ell^{\eps}_{\star}$ the length of the resonator $L^{\eps}$ which minimizes the quantity $\ell\mapsto|R_+^{\eps}(\ell)|$ in a neighbourhood of $\pi(m+1/2)/\om$. Here and in the sequel, $R_+^{\eps}(\ell)$ denote the scattering coefficient introduced in (\ref{decompoMainTermPerturbed}) in a geometry with a resonator of length $\ell>0$. From (\ref{DefRegLong}), (\ref{AsymptoFinalResults1}), we see that $\ell^{\eps}_{\star}$ converges to $\pi(m+1/2)/\om$ as $\eps$ tends to zero. Moreover, (\ref{DefRegLong}), (\ref{AsymptoFinalResults1}) also guarantee that we have $\ell^{\eps}_{\star}<\pi(m+1/2)/\om$ for $\eps$ small enough.
\end{remark}

\noindent Let us illustrate these results. To proceed, we compute numerically the scattering solution $u^{\eps}$ introduced in (\ref{decompoMainTermPerturbed}). We use a P2 finite element method in a domain obtained by truncating $\Om^{\eps}$. On the artificial boundary created by the truncation, a Dirichlet-to-Neumann operator with 15 terms serves as a transparent condition (see more details for example in \cite{Gold82,HaPG98,BoLe11}). Once we have computed $u^{\eps}$, we get easily the scattering coefficients $R^{\eps}_+$, $T^{\eps}$  in the representation (\ref{decompoMainTermPerturbed}). For all the simulations of this article, the wavenumber $\om$ is set to $\om=0.8\pi$. The computations have been made using the library \texttt{Freefem++} \cite{Hech12}.\\[-10pt]
\newline

\begin{figure}[!ht]
\centering
\includegraphics[width=0.48\textwidth]{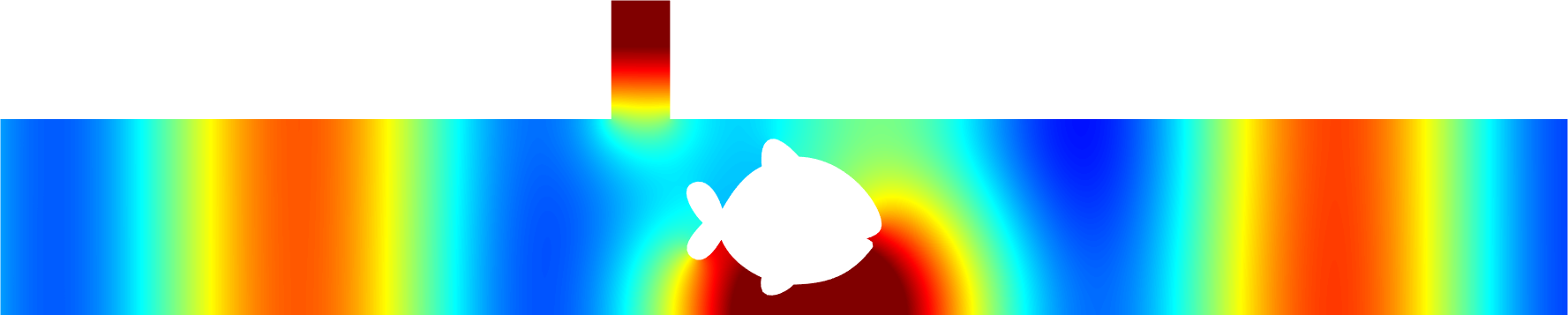}\quad\ \includegraphics[width=0.48\textwidth]{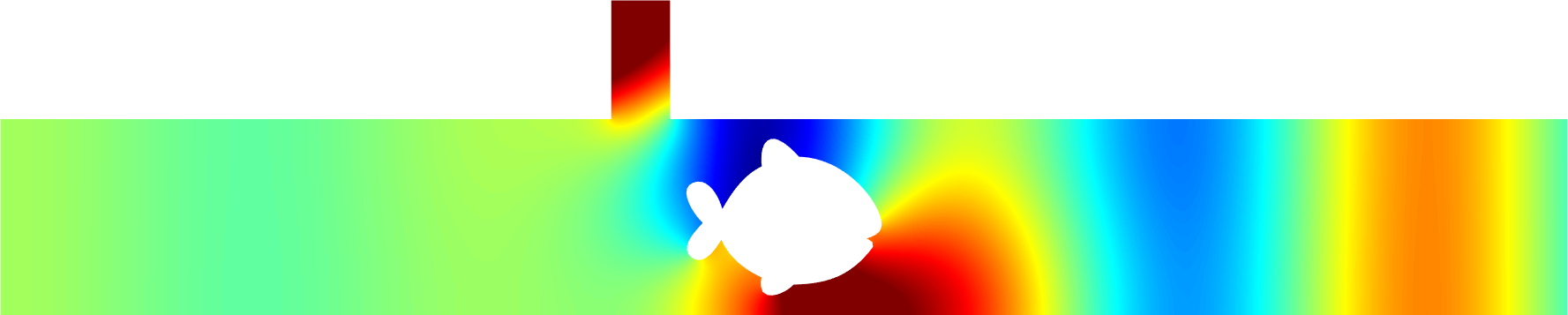}
\caption{Real parts of $u^\eps$ (left) and of $u^\eps-\mrm{w}^+$ (right). The length of the resonator is tuned to get almost zero reflection (Proposition \ref{PropositionZeroRef}). Here $\eps=0.3$.\label{FigFish01}}
\end{figure}

\begin{figure}[!ht]
\centering
\includegraphics[width=0.48\textwidth]{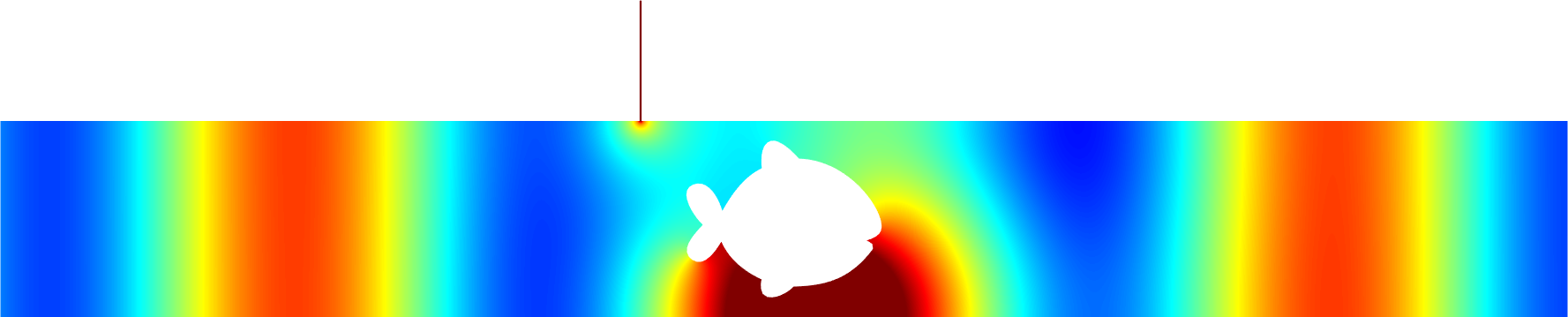}\quad\  \includegraphics[width=0.48\textwidth]{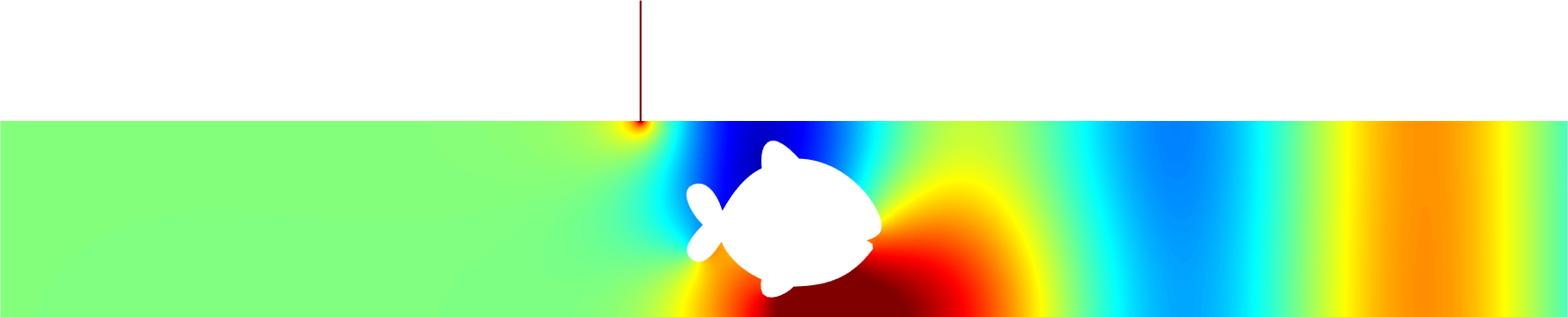}
\caption{Same quantities as in Figure \ref{FigFish01} but with a thinner resonator (here $\eps=0.01$).\label{FigFish001}}
\end{figure}

\begin{figure}[!ht]
\centering
\includegraphics[trim={0cm 0cm 0cm 0cm},clip,height=0.3\textwidth]{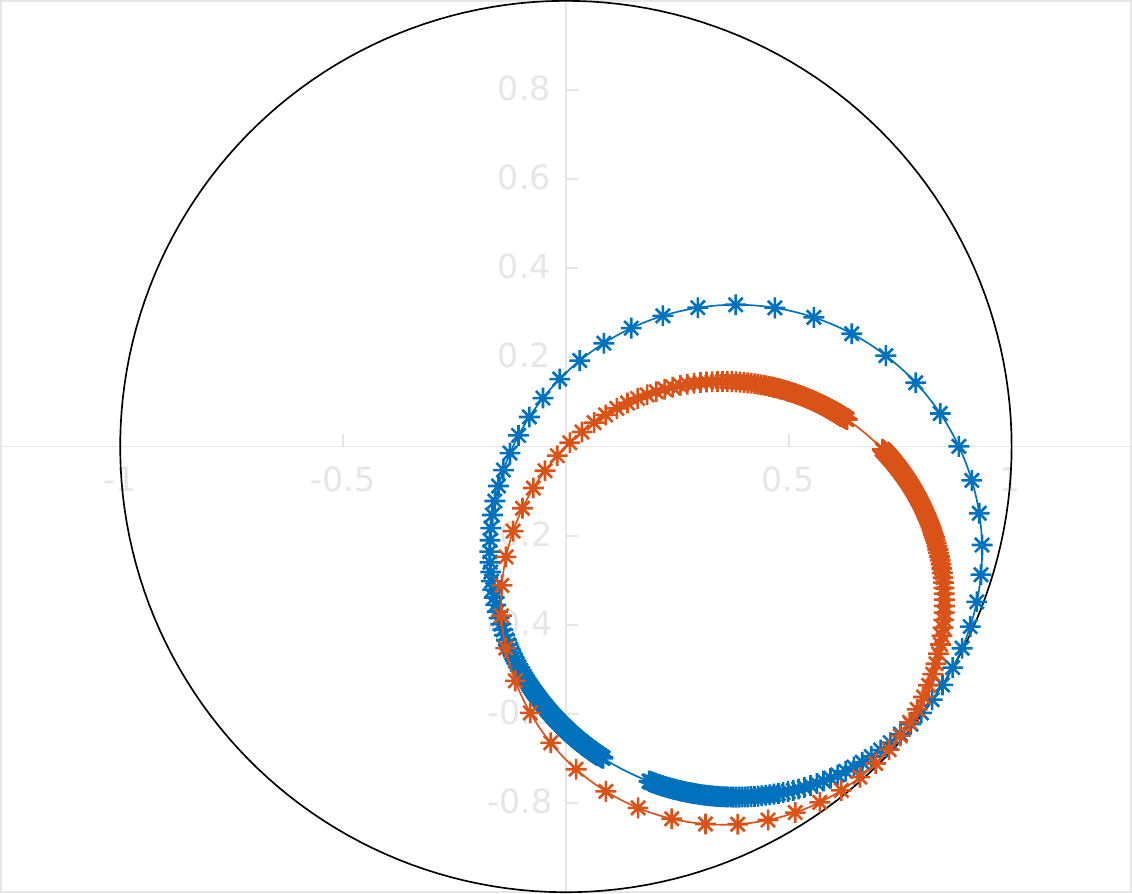}\qquad\raisebox{-1mm}{\includegraphics[trim={0cm 0cm 0cm 0cm},clip,height=0.3\textwidth]{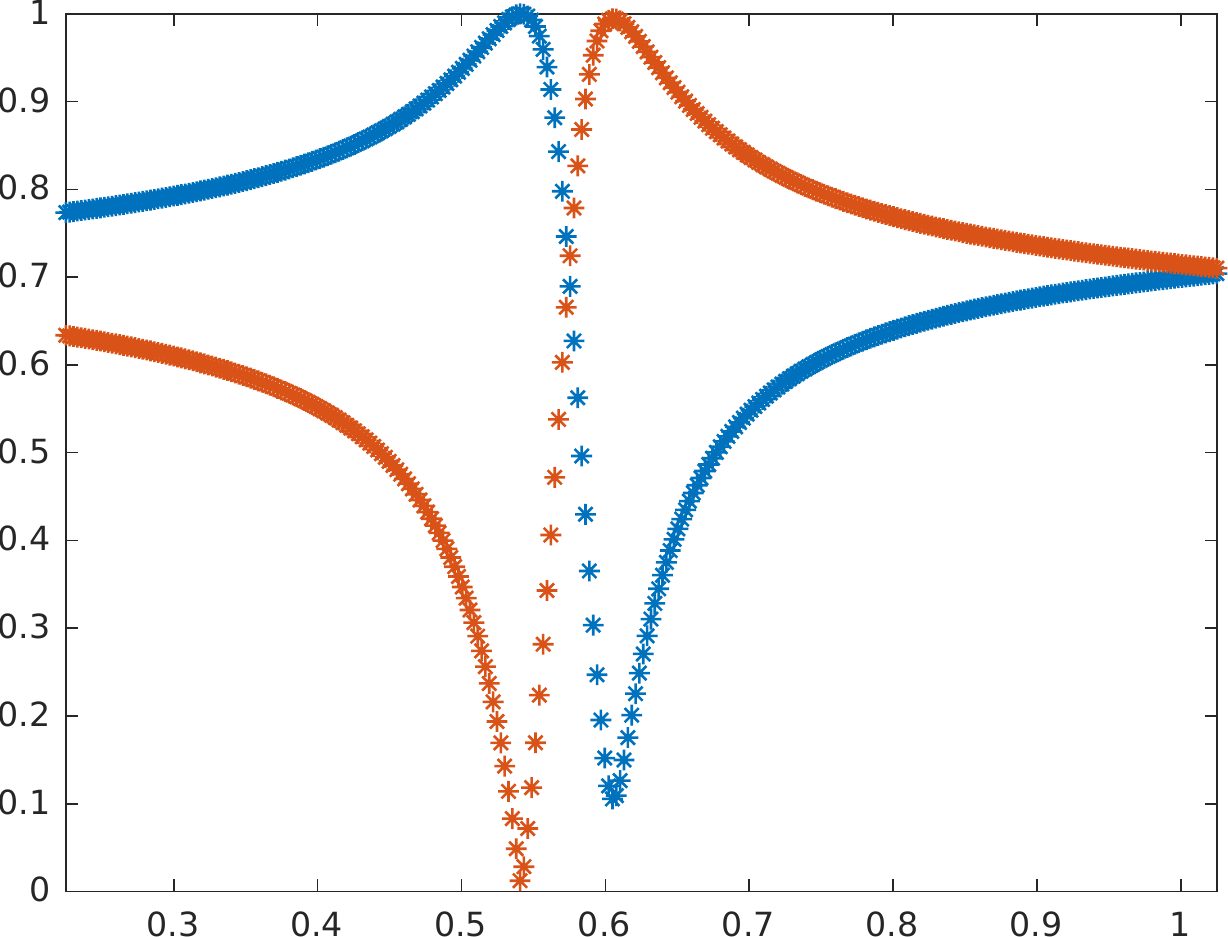}}\\[5pt]
\includegraphics[trim={0cm 0cm 0cm 0cm},clip,height=0.3\textwidth]{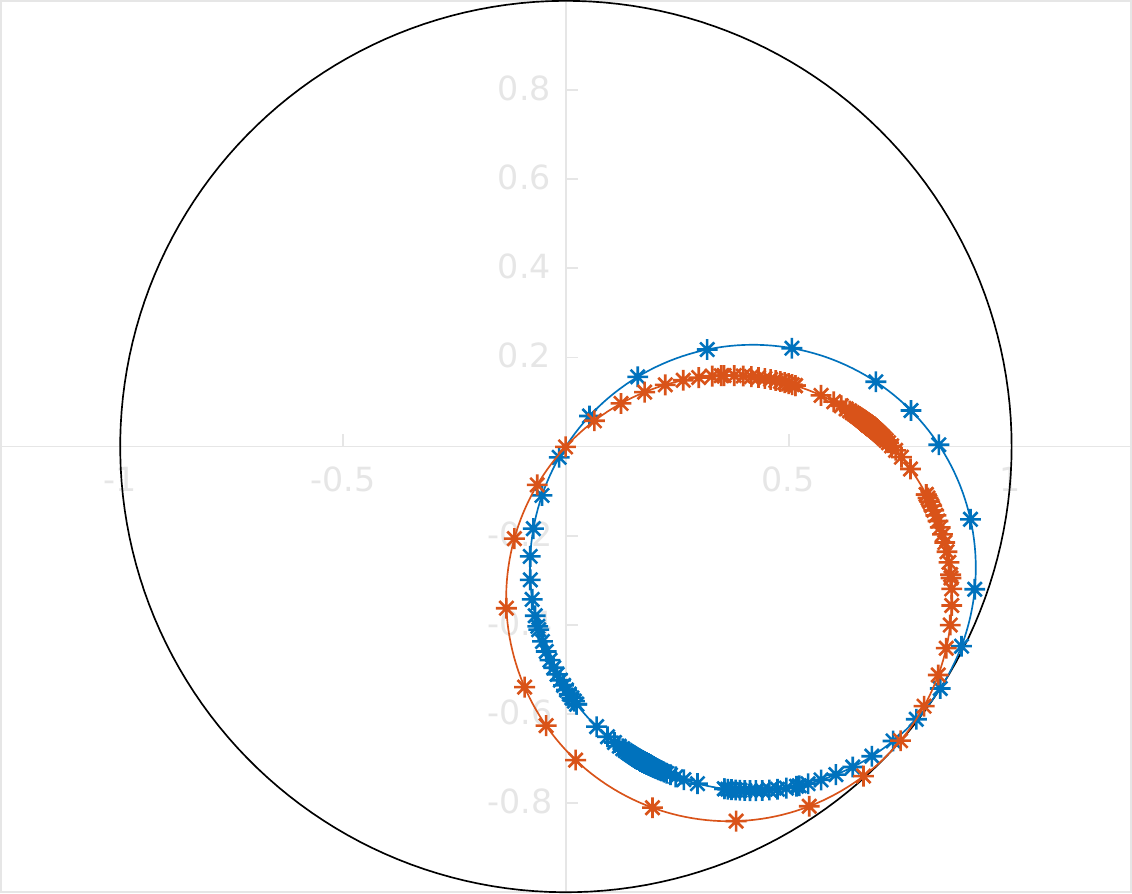}\qquad\raisebox{-1mm}{\includegraphics[trim={0cm 0cm 0cm 0cm},clip,height=0.3\textwidth]{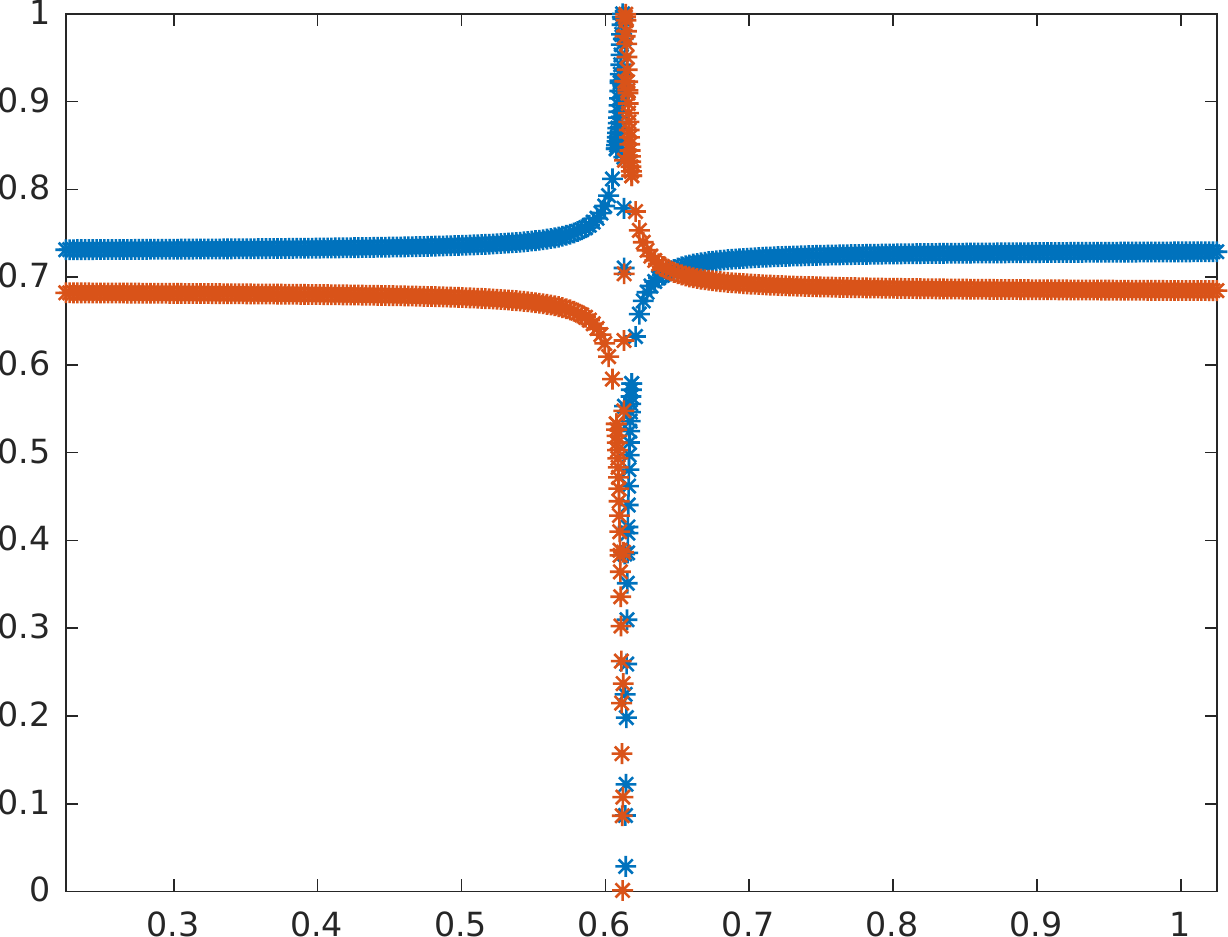}}\\
\caption{Left: curves $\ell\mapsto R^{\eps}_+(\ell)$ (\raisebox{0.3mm}{\textcolor{blue}{\scriptsize{$\rlap{+}{\times}$}}}) and $\ell\mapsto T^{\eps}(\ell)$ (\raisebox{0.3mm}{\textcolor{red}{\scriptsize{$\rlap{+}{\times}$}}}) in the complex plane. The thin coloured circles are the best circles which fit the data. According to the conservation of energy, we have $|R^{\eps}_+(\ell)|^2+|T^{\eps}(\ell)|^2=1$. Therefore the scattering coefficients are located inside the unit disk delimited by the black bold line. Right: curves $\ell\mapsto |R^{\eps}_+(\ell)|$ (\raisebox{0.3mm}{\textcolor{blue}{\scriptsize{$\rlap{+}{\times}$}}}) and $\ell\mapsto |T^{\eps}(\ell)|$ (\raisebox{0.3mm}{\textcolor{red}{\scriptsize{$\rlap{+}{\times}$}}}). We work in the geometries of Figures \ref{FigFish01}, \ref{FigFish001}: $\eps=0.3$ for the first line while $\eps=0.01$ for the second. Here $\ell$ takes values close to $\ell^0=\pi/(2\om)=0.625$. 
\label{FigFishMatSca}}
\end{figure}

\noindent In Figure \ref{FigFish01}--\ref{FigFishMatSca}, we work in a geometry with an obstacle and one resonator. In Figure \ref{FigFish01} ($\eps=0.3$) and in Figure \ref{FigFish001} ($\eps=0.01$), we tune the length of the resonator to get almost zero reflection. Note though the obstacle is rather big, due the well-tuned thin resonator, the scattered field $u^\eps-\mrm{w}^+$ is indeed exponentially decaying in the incoming (left) branch. The position of the resonator has been determined by finding positions $A$ such that numerically identity (\ref{WithoutApprox}) is satisfied. As expected, in Figure \ref{FigFishMatSca} left (blue circles), we remark that the reflection gets even smaller at the particular $\ell^{\eps}_{\star}$ introduced in Remark \ref{RemarkBelowCritical} as $\eps$ tends to zero. On the other hand, in accordance with what is described in Figure \ref{Illustration}, we observe in Figure \ref{FigFishMatSca} right that the variation of the scattering coefficients becomes even faster as $\eps$ tends to zero. As a consequence, when $\eps$ is very small, it is more delicate to tune the length of the resonator to get almost zero reflection. Thus there is a compromise to find between small reflection and robustness with respect to perturbations of the length of the resonator (see also item $vi)$ of Section \ref{SectionConclusion} for a method to improve the quality of the almost zero reflection). Besides, in Figure \ref{FigFishMatSca} right, in accordance with Remark \ref{RemarkBelowCritical}, we note that the length of the resonator such that we observe almost zero reflection converges to $\pi/(2\om)=0.625$ as $\eps$ tends to zero and is smaller than $\pi/(2\om)$.

\begin{figure}[!ht]
\centering
\includegraphics[width=0.48\textwidth]{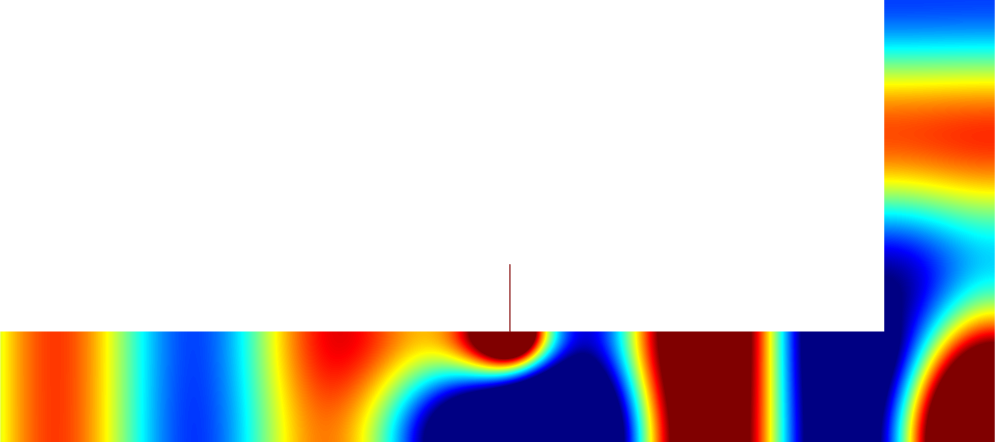}\quad\ \includegraphics[width=0.48\textwidth]{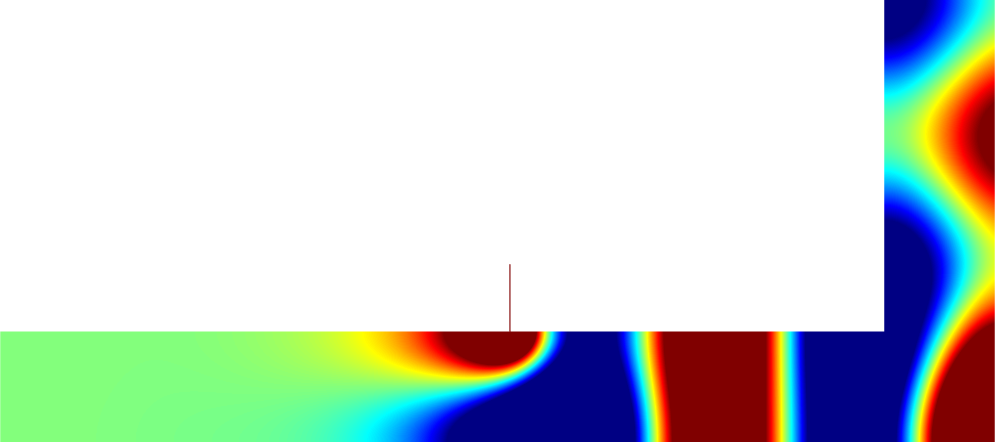}
\caption{Real parts of $u^\eps$ (left) and of $u^\eps-\mrm{w}^+$ (right). The length of the resonator is tuned to get almost zero reflection. Here $\eps=0.01$.\label{FigCoude}}
\end{figure}

\noindent In Figure \ref{FigCoude}, we work in a $L$-shaped waveguide which does not enter strictly the framework introduced in Section \ref{SectionSetting} because it does not coincide with the reference strip $S$ outside of a compact region. However the analysis can be adapted in a straightforward manner and we can find positions and lengths of the resonator to have almost zero reflection. Let us emphasize that the angle between the two branches does not need to be right and can take any value.

\subsection{Additional comments}

\subsubsection{Symmetric geometry} 
Assume that the initial waveguide $\Om$ is symmetric with respect to the $(Oy)$ axis, \textit{i.e.} assume that there holds $\Om=\{(-x,y)\,|\,(x,y)\in\Om\}$. In this case, the function $W^{\pm}$ satisfy $W^+(x,y)=W^-(-x,y)$ for all $(x,y)\in\Om$. As a consequence, if we take $p=0$ and so $A=(0,1)$, there holds $W^+(A)=W^-(A)$. Moreover when $\Om$ is symmetric with respect to the $(Oy)$ axis, we have $R_+=R_-$ so that the identities (\ref{IdentityStructure}) assessed in $A$ simply write
\[
W^+(A)=\overline{W^+(A)}(R_++T).
\]
Then the formulas (\ref{SolutionSystem}) become 
\[
\dsp R^{0}_+=R_+-\cfrac{i(R_++T)}{2\eta+i}\ ,\qquad\qquad T^{0}=T-\cfrac{i(R_++T)}{2\eta+i}.
\]
We deduce that $\eta\mapsto R^{0}_+(\eta)$ passes through zero if and only if $\Re e\,(T/R_+)=0$. But the unitarity of the scattering matrix (see (\ref{RelConsNRJ})) together with the relation $R_+=R_-$ guarantee that $\Re e\,(\overline{R_+}T)=0$, which indeed implies $\Re e\,(T/R_+)=0$. This proves the following statement.
\begin{proposition}\label{PropositionZeroRefSym}
Assume that $R_+\ne0$, $T\ne0$ and that $\Om$ is symmetric with respect to the $(Oy)$ axis. Set $A=(0,1)$. Then there holds $R^{0}_+(\eta)=0$ for some $\eta\in\R$. In this geometry $\Om^{\eps}$, Proposition \ref{PropoAsymOneLig} ensures that one can make $R_{+}^{\eps}$ as small as one wishes by taking $\eps$ small enough and by tuning correctly the length of the resonator $L^{\eps}$.
\end{proposition}
\begin{remark}
Note that contrary to the previous paragraph, here the resonator does not need to be placed ``far'' from the obstacle. Observe also that for another $\eta$, we have $T^0(\eta)=0$.
\end{remark}
\begin{remark}\label{RmkZeroRefExact}
Let us mention that by exploiting the symmetry with respect to the $(Oy)$ axis, we can prove that there is $\eps_0$ small enough such that for all $\eps\in(0;\eps_0]$, the map $\ell\mapsto R^{\eps}_+(\ell)$ passes through zero exactly and not only asymptotically. To show this result, the idea is to decompose $u^{\eps}$ as $u^{\eps}=(u^{\eps}_N+u^{\eps}_D)/2$ where $u^{\eps}_N$, $u^{\eps}_D$ are scattering solutions of problems set in the half-waveguide geometry $\{(x,y)\in\Om^{\eps}\,|\,x<0\}$ with a homogeneous Neumann or Dirichlet boundary condition on the segment $\{(x,y)\in\Om^{\eps}\,|\,x=0\}$. Due to conservation of energy, the corresponding scattering coefficients $R^{\eps}_N$, $R^{\eps}_D$ have modulus equal to one and there holds $R^{\eps}_+=(R^{\eps}_N+R^{\eps}_D)/2$. For $\eps$ small, for $\ell$ varying in a neighbourhood of the resonant length, $R^{\eps}_D$ does not move much on the unit circle while $R^{\eps}_N$ makes a complete lap. As a consequence, there is one $\ell$ such that $ R^{\eps}_+(\ell)=0$. For more details, we refer the reader to \cite{ChNa18}.
\end{remark}

\noindent In Figure \ref{FigStar}, \ref{FigStarMatSca}, we work in a geometry $\Om^{\eps}$ which is symmetric with respect to the $(Oy)$ axis. A thin resonator is located at the abscissa $x=0$. In accordance with Proposition \ref{PropositionZeroRefSym}, we see that we can tune its length to obtain zero reflection. More precisely, though we deliberately choose a quite large $\eps$ ($\eps=0.3$), as explained in Remark \ref{RmkZeroRefExact}, we observe that we can get exact zero reflection, which is not the case in general in absence of symmetry, see e.g. Figure \ref{FigFishMatSca} top left.

\begin{figure}[!ht]
\centering
\includegraphics[width=0.8\textwidth]{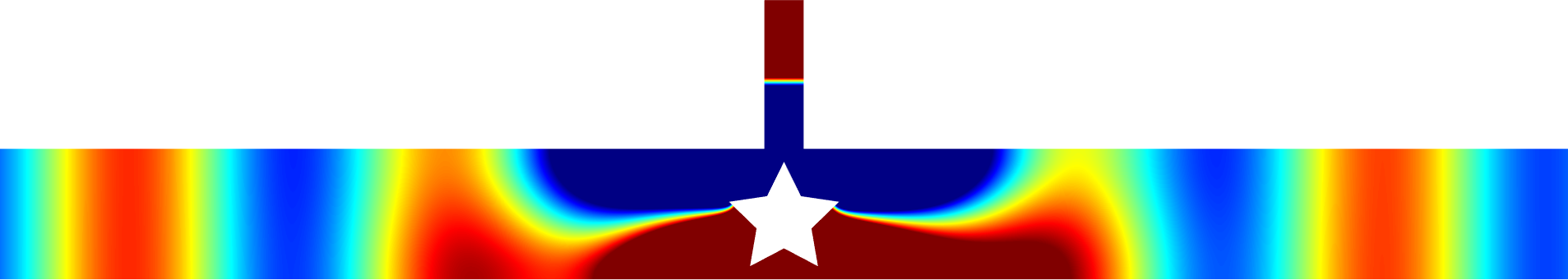}\\[5pt]
\includegraphics[width=0.8\textwidth]{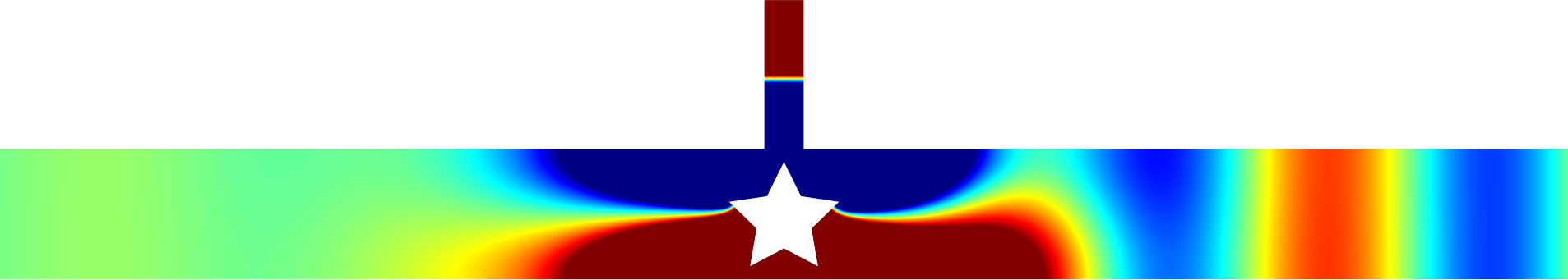}\\
\caption{Real parts of $u^\eps$ (top) and of $u^\eps-\mrm{w}^+$ (bottom). Due to the symmetry of the geometry with respect to the $(Oy)$ axis, we are able to tune the length of the resonator to get exact zero reflection (Remark \ref{RmkZeroRefExact}). Here $\eps=0.3$.\label{FigStar}}
\end{figure}

\begin{figure}[!ht]
\centering
\includegraphics[trim={0cm 0cm 0cm 0cm},clip,width=0.4\textwidth]{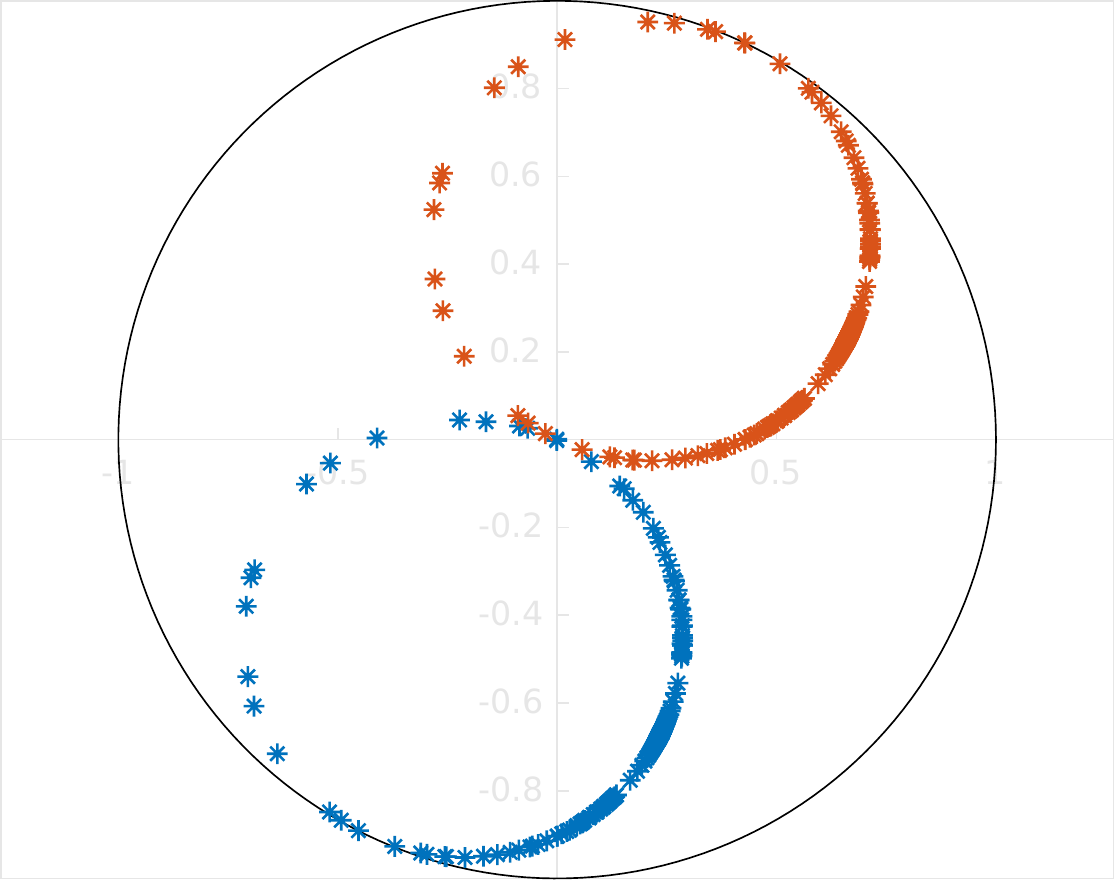}
\caption{Curves $\ell\mapsto R^{\eps}_+(\ell)$ (\raisebox{0.3mm}{\textcolor{blue}{\scriptsize{$\rlap{+}{\times}$}}}) and $\ell\mapsto T^{\eps}(\ell)$ (\raisebox{0.3mm}{\textcolor{red}{\scriptsize{$\rlap{+}{\times}$}}}) in the complex plane in the geometry of Figure \ref{FigStar}.  Here $\ell$ takes values close to $\ell^0=\pi/(2\om)=0.625$ and $\eps=0.3$. Note that the curve for the reflection coefficient passes exactly through zero.
\label{FigStarMatSca}}
\end{figure}

\subsubsection{Zero transmission}\label{ParagZeroT}

In this paragraph, we make a small digression concerning zero transmission, that is $T^{\eps}(\ell)=0$. This situation corresponds to the existence of a standing wave. All the energy of the incident wave is backscattered as if the waveguide were obstructed. This has no link with cloaking but could be interesting for other applications. For example, it can be used to construct waveguides supporting trapped modes, see \cite{ChPa19}. This result will be also useful in the proof of Proposition \ref{PropositionCloakingOneLig} below. Let us forget about the symmetry assumption of $\Om$ made in the previous section and consider a general $A\in\partial\Om$. Starting from the formulas (\ref{SolutionSystem}), we see that asymptotically the transmission coefficient is zero for a certain $\eta$ if and only if 
\begin{equation}\label{CriterionZeroT}
4\eta+i(|W^+(A)|^2+|W^-(A)|^2)=2iW^{+}(A)W^{-}(A)/T.
\end{equation}
At this point, to be rigorous, again we must exclude the case $T=0$. Identity (\ref{CriterionZeroT}) is valid for a certain $\eta\in\R$ if and only if there holds
\begin{equation}\label{identityZeroTransmission}
|W^+(A)|^2+|W^-(A)|^2=2\Re e\,(W^{+}(A)W^{-}(A)/T).
\end{equation}
But using relations (\ref{IdentityStructure}) assessed in $A$, we obtain the two following formulas
\begin{equation}\label{formula1}
2\Re e\,(W^{+}(A)W^{-}(A)/T)=2|W^+(A)|^2+2\Re e\,(W^{+}(A)\overline{W^{-}(A)}R_-/T)
\end{equation}
\begin{equation}\label{formula2}
2\Re e\,(W^{+}(A)W^{-}(A)/T)=2|W^-(A)|^2+2\Re e\,(\overline{W^{+}(A)}W^{-}(A)R_+/T)
\end{equation}
On the other hand, the unitary property of $\mathbb{S}$ imposes that $R_+\overline{T}+\overline{R_-}T=0$ (see (\ref{RelConsNRJ})) so that $R_+/T=-\overline{R_-}/\overline{T}$. Therefore summing (\ref{formula1}) and (\ref{formula2}), we find that (\ref{identityZeroTransmission}) is indeed satisfied. This guarantees that there is $\eta_{\bullet}$ such that 
$T^0(\eta_{\bullet})=0$. Now using again the structure of the scattering matrix, working as \cite[Theorem 5.1]{ChNa20bis} (see also \cite{Lee99}), one can show that if the transmission coefficient passes very close to zero, necessarily it goes exactly through zero. In other words, zero transmission occurs not only asymptotically but also exactly for $\eps$ small enough. We sketch the proof. For a given $\eps$, denote by $\ell^{\eps}_{\bullet}$ the length corresponding to $\eta_{\bullet}$ defined by (\ref{DefRegLong}). Denote also by $I^{\eps}$ the interval of lengths associated with $\eta\in[\eta_{\bullet}-\sqrt{\eps};\eta_{\bullet}+\sqrt{\eps}]$ according to (\ref{DefRegLong}). The idea is to work by contradiction and to assume that $\ell\mapsto T^{\eps}(\ell)$ does not vanish on $I^{\eps}$. In this case, since the scattering matrix is unitary, we have
\begin{equation}\label{ZeroTScaMat}
R_+^{\eps}(\ell)/\overline{R_-^{\eps}(\ell)}=-T^{\eps}(\ell)/\overline{T^{\eps}(\ell)}\,.
\end{equation}
If $\ell\mapsto T^{\eps}(\ell)$ passes very close to zero but not through zero, the right hand side of (\ref{ZeroTScaMat}) runs rapidly on the unit circle on $I^{\eps}$. On the other hand, the left hand side of (\ref{ZeroTScaMat}) converges to $R_+^{0}(\eta_{\bullet})/\overline{R_-^{0}(\eta_{\bullet})}$ ($R_-^{0}(\eta_{\bullet})$ is the main term in the asymptotics of $R_-^{\eps}(\ell^{\eps}_{\bullet})$). This yields a contradiction and ensures that $\ell\mapsto T^{\eps}(\ell)$ passes exactly through zero for $\eps$ small enough. We summarize this result in the following statement.
\begin{proposition}\label{propositionZeroT}
Assume that $T\ne0$. There is $\eps_0>0$ such that for all $\eps\in(0;\eps_0]$, there is $\ell^{\eps}_{\star}$ close to the resonant length $\pi(m+1/2)/\om$, such that $T^{\eps}(\ell^{\eps}_{\star})=0$. Here $T^{\eps}(\ell^{\eps}_{\star})$ stands for the transmission coefficient introduced in (\ref{decompoMainTermPerturbed}) in a geometry with a resonator $L^{\eps}=(p-\eps/2;p+\eps/2)\times[1;1+\ell^{\eps}_{\star})$.
\end{proposition}

\begin{remark}
Note that this result of zero transmission is different from the result of zero reflection for two reasons. First, it holds without assumption on the position of the resonator (there is no assumption on $p$). Second, it holds exactly whereas zero reflection in general is obtained only asymptotically. Thus we can say that zero transmission is simpler to achieve than zero reflection with this mechanism.
\end{remark}

\noindent In Figure \ref{FigZeroT} we display the imaginary part of $u^\eps$ in the reference strip with one resonator. The length of the resonator is tuned to get exact zero transmission which is possible according to Proposition \ref{propositionZeroT}. We display the imaginary part of $u^\eps$ because here $R^\eps_+=-1$ so that the field is purely imaginary at infinity. It is remarkable to see the influence of the resonator. Though it is very thin ($\eps=0.01$), it stops the propagation of the wave and everything happens as if the waveguide were obstructed. Note that we also observe in Figures \ref{FigFishMatSca}, \ref{FigStarMatSca} that the curve for the transmission coefficient always passes exactly through zero (even when $\eps$ is not that small).

\begin{figure}[!ht]
\centering
\includegraphics[width=0.8\textwidth]{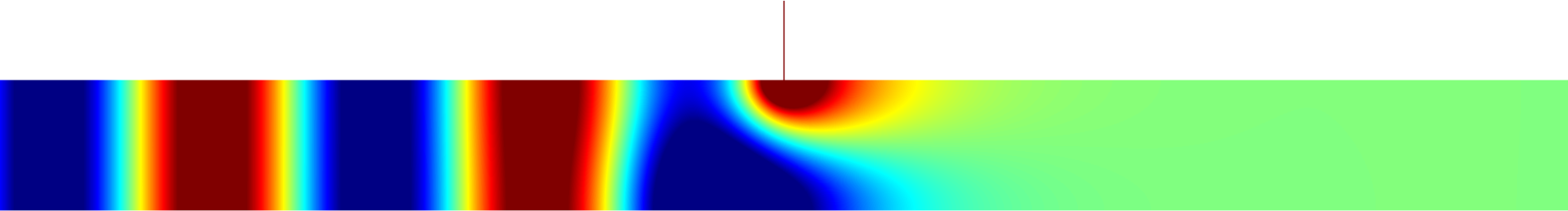}
\caption{Imaginary part of $u^\eps$. The length of the resonator is tuned to get exact zero transmission (Proposition \ref{propositionZeroT}). Here $\eps=0.01$.\label{FigZeroT}}
\end{figure}

\section{Phase shifter}\label{SectionPhaseShift}
In the previous section, we explained how to annihilate the reflection. When the reflection is zero, due to conservation of energy, the modulus of the transmission coefficient is equal to one. But in general the phase of this transmission coefficient can be non zero and we have to compensate this phase shift to cloak completely the obstacle. In this section, we show how to design geometries, that we call phase shifters, where the reflection coefficient is zero and the transmission coefficient can take any value on the unit circle. To proceed, we work in the unperturbed strip with two thin resonators. We propose two ways of establishing the result: one where we use twice successively the asymptotic analysis of \S\ref{paragraphAsympto} with one thin resonator and another where we compute an asymptotic expansion directly with the two thin resonators.

\subsection{Successive asymptotic analysis}\label{paragraphSuccessive}

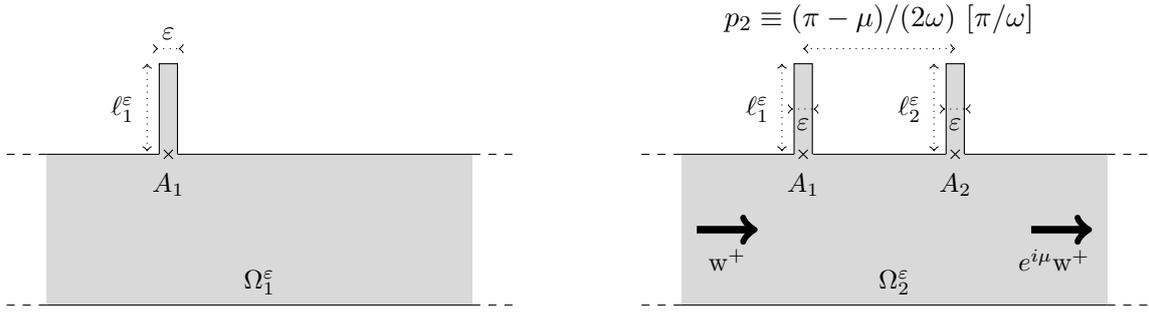
\begin{figure}[!ht]
\centering
\begin{tikzpicture}[scale=2]
\draw[fill=gray!30,draw=none](-1.4,0) rectangle (1.4,1);
\draw (-1.4,0)--(1.4,0);
\draw (-1.4,1)--(1.4,1);
\draw[dotted,<->] (-0.74,1.02)--(-0.74,1.6);
\draw[dotted,>-<] (-0.7,1.7)--(-0.5,1.7);
\node at (-0.6,1.8){\small $\eps$};
\node at (-0.9,1.3){\small $\ell^\eps_1$};
\node at (-0.6,0.8){\small $A_1$};
\draw[fill=gray!30,draw=none](-0.66,0.9) rectangle (-0.54,1.6);
\draw (-0.66,1)--(-0.66,1.6)--(-0.54,1.6)-- (-0.54,1);
\draw[dashed] (-1.4,0)--(-1.7,0);
\draw[dashed] (-1.4,1)--(-1.7,1);
\draw[dashed] (1.4,0)--(1.7,0);
\draw[dashed] (1.4,1)--(1.7,1);
\draw (-0.63,0.97)--(-0.57,1.03);
\draw (-0.63,1.03)--(-0.57,0.97);
\node at (0,0.15){\small $\Om^{\eps}_1$};
\end{tikzpicture}\qquad\qquad\begin{tikzpicture}[scale=2]
\draw[fill=gray!30,draw=none](-1.4,0) rectangle (1.4,1);
\draw (-1.4,0)--(1.4,0);
\draw (-1.4,1)--(1.4,1);
\draw[fill=gray!30,draw=none](-0.66,0.9) rectangle (-0.54,1.6);
\draw[dotted,<->] (-0.74,1.02)--(-0.74,1.6);
\draw[dotted,>-<] (-0.7,1.3)--(-0.5,1.3);
\node at (-0.9,1.3){\small $\ell^\eps_1$};
\node at (-0.6,0.8){\small $A_1$};
\node at (-0.6,1.2){\small $\eps$};
\draw (-0.66,1)--(-0.66,1.6)--(-0.54,1.6)-- (-0.54,1);
\draw (-0.63,0.97)--(-0.57,1.03);
\draw (-0.63,1.03)--(-0.57,0.97);
\begin{scope}[shift={(1,0)}]
\draw[fill=gray!30,draw=none](-0.66,0.9) rectangle (-0.54,1.6);
\draw[dotted,<->] (-0.74,1.02)--(-0.74,1.6);
\draw[dotted,>-<] (-0.7,1.3)--(-0.5,1.3);
\node at (-0.6,1.2){\small $\eps$};
\node at (-0.9,1.3){\small $\ell^\eps_2$};
\node at (-0.6,0.8){\small $A_2$};
\draw (-0.66,1)--(-0.66,1.6)--(-0.54,1.6)-- (-0.54,1);
\draw (-0.63,0.97)--(-0.57,1.03);
\draw (-0.63,1.03)--(-0.57,0.97);
\end{scope}
\draw[->,line width=1mm] (-1.3,0.5)--(-0.9,0.5);
\draw[<-,line width=1mm] (1.3,0.5)--(0.9,0.5);
\draw[dashed] (-1.4,0)--(-1.7,0);
\draw[dashed] (-1.4,1)--(-1.7,1);
\draw[dashed] (1.4,0)--(1.7,0);
\draw[dashed] (1.4,1)--(1.7,1);
\node at (0,0.15){\small $\Om^{\eps}_2$};
\node at (-1.1,0.3){\small $\mrm{w}^+$};
\node at (1.05,0.3){\small $e^{i\mu}\mrm{w}^+$};
\draw[dotted,<->] (-0.6,1.7)--(0.4,1.7);
\node at (-0.1,1.9){$ p_2\equiv (\pi-\mu)/(2\om)$ $[\pi/\om]$};
\end{tikzpicture}
\caption{Geometry of $\Om_1^\eps$ (left) and schematic picture of the final phase shifter $\Om^\eps_{2}$ (right). \label{StripWithLigaments}} 
\end{figure}

As already said, we start from the unperturbed strip $S=\R\times(0;1)$. In this case, we simply have $W^{\pm}=\mrm{w}^{\pm}$ in (\ref{Field1})--(\ref{Field2}) so that $R_+=0$ and $T=1$. Set $A_1:=(p_1,1)$ with $p_1=0$ and, following (\ref{DefRegLong}), add at $A_1$ a resonator of length 
\[
\ell^{\eps}_1= \pi(m+1/2)/\om+\eps(\eta_1/\om-C_{\Xi}-\Re e\,\Gamma_1-\pi^{-1}|\ln\eps|)
\]
for a certain $m\in\N$ and $\eta_1\in\R$. Here $\Gamma_1$ is defined as in (\ref{DefGammaM}) from a $\mu_1$ satisfying (\ref{DefGamma}) in $S$. We set $\Om^{\eps}_1:=S\cup (-\eps/2;\eps/2)\times[1;1+\ell^{\eps}_1)$ (see Figure \ref{StripWithLigaments} left) and denote by $R_{1+}^{\eps}$, $T^{\eps}_1$ the scattering coefficients of (\ref{decompoMainTermPerturbed}) in $\Om^{\eps}_1$. According to Proposition \ref{PropoAsymOneLig}, when $\eps$ tends to zero, we have the expansions  
\begin{equation}\label{ScaLim}
\begin{array}{|lcl}
R_{1+}^{\eps}=R_{1+}^{0}+o(1)\\[4pt]
T^{\eps}_{1}=T_{1}^{0}+o(1)
\end{array}\qquad
\mbox{ with }\qquad
R_{1+}^{0}:=-\cfrac{i}{2\eta_1+i}\,,\qquad T_{1}^{0}:=1-\cfrac{i}{2\eta_1+i}\,.
\end{equation}
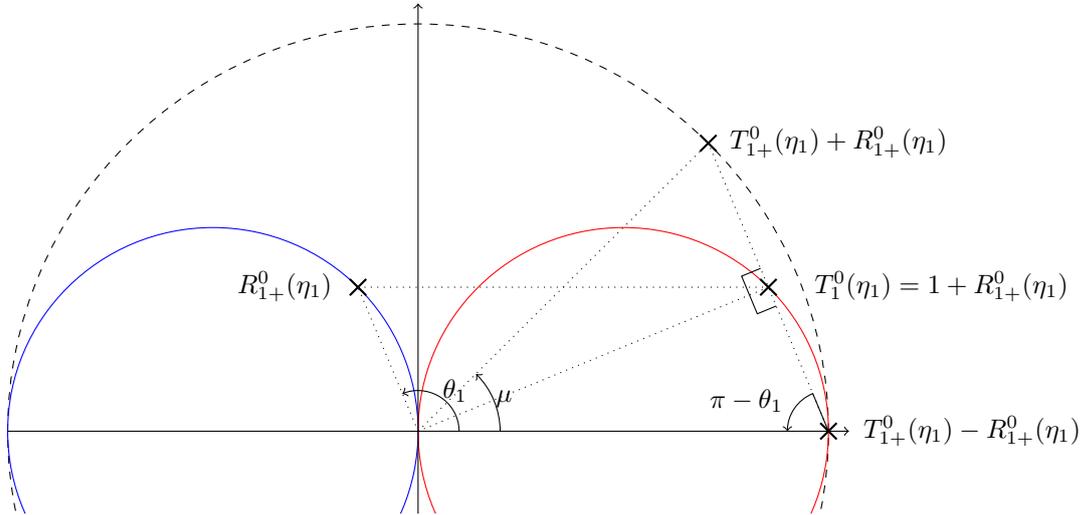
\begin{figure}[!ht]
\centering
\begin{tikzpicture}[scale=2.7]
\begin{scope}
\clip (-2,-0.4) -- (3.5,-0.4) -- (3.5,2.1) -- (-2,2.1) -- cycle;
\draw[dashed] (0,0) circle (2);
\draw[blue, thin] (-1,0) circle (1);
\draw[red, thin] (1,0) circle (1);
\draw[->] (0,-2.05)--(0,2.1);
\draw[->] (-2.05,0)--(2.1,0);
\begin{scope}[shift={(-0.292893219,0.707106781 )}]
\draw[thick] (0.04,0.04)--(-0.04,-0.04);
\draw[thick] (0.04,-0.04)--(-0.04,0.04);
\end{scope}
\begin{scope}[shift={(2,0)}]
\draw[thick] (0.04,0.04)--(-0.04,-0.04);
\draw[thick] (0.04,-0.04)--(-0.04,0.04);
\end{scope}
\begin{scope}[shift={(2,0)}]
\begin{scope}[shift={(-0.292893219,0.707106781 )}]
\draw[thick] (0.04,0.04)--(-0.04,-0.04);
\draw[thick] (0.04,-0.04)--(-0.04,0.04);
\end{scope}
\end{scope}
\begin{scope}[shift={(2-0.585786438,1.41421356 )}]
\draw[thick] (0.04,0.04)--(-0.04,-0.04);
\draw[thick] (0.04,-0.04)--(-0.04,0.04);
\end{scope}
\begin{scope}[shift={(2-0.292893219,0.707106781)},rotate=202.5]
\draw (0,0.1)--(0.1,0.1)--(0.1,0);
\end{scope}
\begin{scope}[shift={(2-0.292893219,0.707106781)},rotate=112.5]
\draw (0,0.1)--(0.1,0.1)--(0.1,0);
\end{scope}
\node at (-0.65,0.71){\small $R_{1+}^{0}(\eta_1)$};
\node at (2.55,0.71){\small $T_{1}^{0}(\eta_1)=1+R_{1+}^{0}(\eta_1)$};
\draw[dotted] (2,0)--+(-0.585786438,1.41421356 );
\draw[dotted] (0,0)--+(2-0.292893219,0.707106781);
\draw[dotted] (0,0)--+(-0.292893219,0.707106781);
\node at (2.7,0){\small $T_{1+}^{0}(\eta_1)-R_{1+}^{0}(\eta_1)$};
\node at (2.05,1.42){\small $T_{1+}^{0}(\eta_1)+R_{1+}^{0}(\eta_1)$};
\draw[->] (0.2,0) arc (0.2:112.5:0.2);
\node at (0.18,0.2){\small $\theta_1$};
\draw[->](2,0) -- +(112.5:0.2) arc (112.5:180:0.2);
\node at (2-0.4,0.15){\small $\pi-\theta_1$};
\draw[->] (0.4,0) arc (0.4:45:0.4);
\node at (0.42,0.16){\small $\mu$};
\draw[dotted] (0,0)--(2-0.585786438,1.41421356 );
\draw[dotted] (-0.292893219,0.707106781)--(2-0.292893219,0.707106781);
\end{scope}
\end{tikzpicture}
\caption{Sets $\{R_{1+}^{0}(\eta_1),\,\eta_1\in\R\}$ (left blue circle) and $\{T_{1}^{0}(\eta_1),\,\eta_1\in\R\}$ (right red circle) in the complex plane where $T_{1}^{0}(\eta_1)$, $R_{1+}^{0}(\eta_1)$ are defined in (\ref{ScaLim}). The dashed circle represents the unit circle. The points of affix $T_{1+}^{0}(\eta_1)\pm R_{1+}^{0}(\eta_1)$ correspond to the approximate transmission coefficients we get when we add a second resonator and tune its length to obtain almost zero reflection (see (\ref{quotient1})--(\ref{quotient2})). \label{IllustrationAsy}}
\end{figure}

\noindent In particular, we remark that there holds $T_{1}^{0}=1+R_{1+}^{0}$. The set $\{R_{1+}^{0}(\eta_1),\,\eta_1\in\R\}$ (resp. $\{T_{1}^{0}(\eta_1),\,\eta_1\in\R\}$) coincides with $\mathscr{C}(-1/2,1/2)\setminus\{0\}$ (resp. $\mathscr{C}(1/2,1/2)\setminus\{0\}$) (see the illustration of Figure \ref{IllustrationAsy}). Observe that we have the parametrisation
\[
\mathscr{C}(-1/2,1/2)\setminus\{0\}=\{-\cos\theta_1e^{i\theta_1}\,|\,\theta_1\in(\pi/2;3\pi/2)\}.
\]
Assume that $\ell^{\eps}_1>0$ has been fixed so that $R_{1+}^{\eps}\approx -\cos\theta_1e^{i\theta_1}$ for a certain $\theta_1\in(\pi/2;3\pi/2)$. The idea of our approach is to consider this resonator as a fixed given obstacle and to play with a second resonator. More precisely, set 
$A_2:=(p_2,1)$ with $p_2>0$ and, following (\ref{DefRegLong}), add at $A_2$ a resonator of length 
\[
\ell^{\eps}_2= \pi(m+1/2)/\om+\eps(\eta_2/\om-C_{\Xi}-\Re e\,\Gamma_2-\pi^{-1}|\ln\eps|)
\]
for a certain $\eta_2\in\R$. Here $\Gamma_2$ is defined as in (\ref{DefGammaM}) from a $\mu_2$ satisfying (\ref{DefGamma}) in $\Om_1^{\eps}$. In particular, a priori there holds $\Gamma_2\ne\Gamma_1$.
Define $\Om^{\eps}_2:=\Om^{\eps}_1\cup (p_2-\eps/2;p_2+\eps/2)\times[1;1+\ell^{\eps}_2)$ and denote by $R_{2+}^{\eps}$, $T^{\eps}_2$ the scattering coefficients of (\ref{decompoMainTermPerturbed}) in $\Om^{\eps}_2$. Using again the result of Proposition \ref{PropoAsymOneLig} which yields the asymptotic expansion of the scattering coefficients in presence of one thin resonator, for $\eps$ small, we get the expansions 
\begin{equation}\label{ScaLim2}
\begin{array}{c}
R_{2+}^{\eps}=R_{2+}^0+o(1)\qquad \mbox{ with }\qquad R_{2+}^{0}:=R_{1+}^{\eps}-\cfrac{2i(W^{+}_1(A_2))^2}{4\eta_2+i(|W^+_1(A_2)|^2+|W^-_1(A_2)|^2)}\\[14pt]
T^{\eps}_{2}=T_{2}^0+o(1)\qquad \mbox{ with }\qquad T_{2}^{0}:=T_{1}^{\eps}-\cfrac{2iW^{+}_1(A_2)W^{-}_1(A_2)}{4\eta_2+i(|W^+_1(A_2)|^2+|W^-_1(A_2)|^2)}\,.
\end{array}
\end{equation}
Here $W^\pm_1$ are the functions defined in (\ref{Field1}) in the geometry $\Om^{\eps}_1$. Using the approximation results $R_{1+}^{\eps}\approx R_{1+}^{0}$, $T^{\eps}_{1}\approx T_{1}^{0}$ and making the assumption that $p_2$ is sufficiently large so that we can neglect evanescent terms, we get 
$W^+_1(A_2)\approx T_1^0\,e^{i\om p_2}$ and $W^-_1(A_2)\approx e^{-i\om p_2}+R_{1+}^0\,e^{i\om p_2}$ (note that since $\Om_1^{\eps}$ is symmetric with respect to the vertical axis, there holds $R_{1+}^0=R_{1-}^0$). Then according to identity (\ref{RelationToGetZeroRRight}), we find that for $p_2$ such that
\begin{equation}\label{RelationSecond}
\cos(2\om p_2+\theta_1)=\cos\theta_1,
\end{equation}
the curve $\eta_2\mapsto R_{2+}^{0}(\eta_2)$ passes very close to zero. For  $p_2$ as in (\ref{RelationSecond}) and $\eta_2$ such that $R_{2+}^{0}(\eta_2)\approx 0$, using (\ref{ScaLim2}), we get
\begin{equation}\label{SeveralApprox}
T_{2}^{\eps}\approx T_{1}^{\eps}-\cfrac{W^-_1(A_2)}{W^+_1(A_2)}\,R_{1+}^{\eps}\approx T_{1}^{0}-\cfrac{W^-_1(A_2)}{W^+_1(A_2)}\,\,R_{1+}^{0}.
\end{equation}
Notice that (\ref{RelationSecond}) holds if $2\om p_2\equiv0$ $[2\pi]$ or if $2\om p_2\equiv -2\theta_1$ $[2\pi]$ (we use the notation $2\om p_2\equiv 0$ $[2\pi]$ to say that $2\om p_2$ is congruent to $0$ modulo $2\pi$). Let us consider each situation separately.\\
\newline
$\star$ When $2\om p_2\equiv0$ $[2\pi]$ and $\theta_1\ne\pi$, using that $R_{1+}^{0}=-\cos\theta_1 e^{i\theta_1}$ and that $T_{1}^{0}=1+R_{1+}^{0}$, we find
\begin{equation}\label{quotient1}
\cfrac{W^-_1(A_2)}{W^+_1(A_2)}\approx 
\cfrac{e^{-i\om p_2}+R_{1+}^0\,e^{i\om p_2}}{T_1^0\,e^{i\om p_2}}=\cfrac{1+R_{1+}^0\,e^{2i\om p_2}}{e^{2i\om p_2}+R_{1+}^0\,e^{2i\om p_2}}=1.
\end{equation}
As a consequence, from (\ref{SeveralApprox}), this yields $T_{2}^{\eps}\approx T_{1}^{0}-R_{1+}^{0}=1$. This is not interesting for our purpose because we cannot control the phase.\\
\newline
$\star$ When $2\om p_2\equiv -2\theta_1$ $[2\pi]$ and $\theta_1\ne\pi$, the calculus is different. More precisely, we find
\begin{equation}\label{quotient2}
\begin{array}{rcl}
\cfrac{W^-_1(A_2)}{W^+_1(A_2)}\approx \cfrac{1+R_{1+}^0\,e^{2i\om p_2}}{e^{2i\om p_2}+R_{1+}^0\,e^{2i\om p_2}}&=&\cfrac{1-\cos\theta_1(\cos\theta_1-i\sin\theta_1)}{\cos(2\theta_1)-i\sin(2\theta_1)-\cos\theta_1(\cos\theta_1-i\sin\theta_1)}\\[12pt]
&=&\cfrac{1-\cos\theta_1(\cos\theta_1-i\sin\theta_1)}{-(1-\cos\theta_1(\cos\theta_1-i\sin\theta_1))}=-1.
\end{array}
\end{equation}
Then from (\ref{SeveralApprox}) we obtain $T_{2}^{\eps}\approx T_{1}^{0}+R_{1+}^{0}= 1-2\cos\theta_1 e^{i\theta_1}$. Writing that 
\[
1-2\cos\theta_1 e^{i\theta_1} = -\cos(2\theta_1)-i\sin(2\theta_1)=e^{i(2\theta_1-\pi)},
\]
we see that we can get $T_{2}^{\eps}\approx e^{i\mu}$ for any $\mu\in(0;2\pi)\setminus\{\pi\}$. To proceed, it suffices to take $\theta_1=(\mu+\pi)/2$ and so $p_2\equiv (\pi-\mu)/(2\om)$ $[\pi/\om]$. Thus these settings allow us to obtain the phase shifter we were looking for. We summarize the construction in the following procedure. 

\begin{procedure}\label{ProcePhaseShift1}
Assume that we want to get a phase shifter with a transmission coefficient approximately equal to $e^{i\mu}$ for some $\mu\in(0;2\pi)\setminus\{\pi\}$.\\
1) Place a thin resonator at $A_1=(0,1)$ and tune its length to get a reflection coefficient equal (up to a small error) to $-\cos\theta_1e^{i\theta_1}$ with $\theta_1:=(\mu+\pi)/2$.\\
2) Then place a second resonator at $A_2=(p_2,1)$ with $p_2>0$ large enough such that $p_2\equiv (\pi-\mu)/(2\om)$ $[\pi/\om]$. Tune its length to get almost zero reflection. According to the above analysis, this yields a transmission coefficient approximately equal to $e^{i\mu}$. 
\end{procedure}

\begin{remark}
If one wishes to obtain a phase shifter with a transmission coefficient approximately equal to $-1$, one can impose $T_{2}^{\eps}\approx e^{i\mu}$ with $\mu$ as close as desired to $\pi$. However to be rigorous, we should exclude the case $\mu=\pi$ in the calculus (\ref{quotient2}). 
\end{remark}

\subsection{Asymptotic analysis with two thin outer resonators}

In this section we compute an asymptotic expansion of the scattering coefficients in the geometry $\Om_2^{\eps}$ depicted in Figure \ref{StripWithLigaments} right with two thin outer resonators. The approach is similar to the one of \S\ref{paragraphAsympto} and we use the same notation. We simply outline the main differences. For the sake of generality, first we work with a waveguide $\Om$ which may contain an obstacle, \textit{i.e.} we work with a waveguide $\Om$ as in (\ref{MainPb}).\\
\newline
In $\Om$, we consider the expansion (\ref{AnsatzWaveguide}) while in the resonator $L_j^{\eps}$, for $j=1,2$, we work with the ansatz
\[
u^\eps(x,y)=\eps^{-1}v^{-1}_j(y)+v^0_j(y)+\dots\quad\mbox{\rm in }L^\eps_j.
\]
We take $\ell_j^{\eps}=\pi(m+1/2)/\om+\eps\ell_j'$ with $m\in\N$ and $\ell_j'\in\R$. Then we find 
\[
v^{-1}_j(y)=a_j {\bf v}_j(y)\qquad\mbox{ with }\qquad a_j\in\mathbb{C},\ {\bf v}_j(y)=\sin(\omega (y-1)).
\]  
We still denote by $R^{0}_+,\,T^0$ the first terms in the asymptotics of $R^{\eps}_+,\,T^\eps$ so that 
\[
R^{\eps}_+=R^{0}_++\dots\qquad\mbox{ and }\qquad T^\eps=T^0+\dots\,.
\]
With two resonators, the formulas (\ref{equation1}) turn into 
\begin{equation}\label{equation1Two}
\begin{array}{|l}
R^{0}_+=R_++i(a_1W^{+}(A_1)+a_2W^{+}(A_2))/2\\[4pt]
T^{0}=T+i(a_1W^{-}(A_1)+a_2W^{-}(A_2))/2.
\end{array}
\end{equation}
On the other hand, decomposition (\ref{SecondDecompo}) becomes 
\begin{equation}\label{SecondDecompoTwo}
u^0=W^++a_1\om\gamma_1+a_2\om\gamma_2
\end{equation}
where for $j=1,2$, $\gamma_j$ is the outgoing function such that
\begin{equation}\label{DefGammai}
\begin{array}{|rcll}
\Delta \gamma_j+\om^2\gamma_j&=&0&\mbox{ in }\Om\\
\partial_\nu\gamma_j&=&\delta_{A_j}&\mbox{ on }\partial\Om.
\end{array}
\end{equation}
Denote by $\Gamma_j$ the constant behaviour of $\gamma_j$ at $A_j$, that is the constant such that $\gamma_j$ behaves as 
\begin{equation}\label{DefGammaMi}
\gamma_j(x,y)= \frac{1}{\pi}\ln \frac{1}{r^{A_j}}+\Gamma_j+O(r^{A_j})\quad \mbox{ when }r^{A_j}:=((x-p_j)^2+(y-1)^2)^{1/2}\rightarrow0^+.
\end{equation}
Lemma \ref{lemmaRelConstantsTwo} in Appendix guarantees that $\gamma_1(A_2)=\gamma_2(A_1)$. We denote by $\tilde{\Gamma}$ the value of this constant. Then equality  (\ref{BoundaryCondition1}) becomes 
\[
\begin{array}{rcl}
v^0_1(1)&=&W^+(A_1)+a_1\om\,(\pi^{-1}|\ln\eps|+C_{\Xi}+\Gamma_1)+a_1\om\tilde{\Gamma}\\[3pt]
v^0_2(1)&=&W^+(A_2)+a_2\om\,(\pi^{-1}|\ln\eps|+C_{\Xi}+\Gamma_2)+a_2\om\tilde{\Gamma}.
\end{array}
\]
Writing the compatibility conditions as in (\ref{CompatibilityCondition}),  we get
\begin{equation}\label{Part2SystemTwo}
\begin{array}{rcl}
0&=&W^+(A_1)+a_1\om\,(\pi^{-1}|\ln\eps|+C_{\Xi}+\Gamma_1+\ell'_1)+a_1\om\tilde{\Gamma} =0\\[3pt]
0&=&W^+(A_2)+a_2\om\,(\pi^{-1}|\ln\eps|+C_{\Xi}+\Gamma_2+\ell'_2)+a_2\om\tilde{\Gamma} =0.
\end{array}
\end{equation}
Using that $C_{\Xi}\in\R$, $\Im m\,(\om\Gamma_j)=(|W^+(A_j)|^2+|W^-(A_j)|^2)/4$ (Lemmas \ref{LemmaCReal} and \ref{lemmaRelConstants}) and gathering (\ref{equation1Two}) and (\ref{Part2SystemTwo}), we obtain the system
\begin{equation}\label{system1Two}
\begin{array}{|l}
R^{0}_+=R_++i(a_1W^{+}(A_1)+a_2W^{+}(A_2))/2\\[4pt]
T^{0}=T+i(a_1W^{-}(A_1)+a_2W^{-}(A_2))/2\\[4pt]
a_1(\eta_1+i\,(|W^+(A_1)|^2+|W^-(A_1)|^2)/4)+a_2\om\tilde{\Gamma} =-W^+(A_1)\\[3pt]
a_2(\eta_2+i\,(|W^+(A_2)|^2+|W^-(A_2)|^2)/4)+a_1\om\tilde{\Gamma} =-W^+(A_2),
\end{array}
\end{equation}
with
\begin{equation}\label{defEtaTwo}
\begin{array}{|l}
\eta_1:=\om(\pi^{-1}|\ln\eps|+C_{\Xi}+\Re e\,\Gamma_1+\ell'_1)\\[3pt]
\eta_2:=\om(\pi^{-1}|\ln\eps|+C_{\Xi}+\Re e\,\Gamma_2+\ell'_2).
\end{array}
\end{equation}
This ends the asymptotic analysis with two resonators. We summarize the results in the following proposition. 
\begin{proposition}\label{PropoAsymTwoLig}
Assume that
\[
\begin{array}{|l}
\ell^{\eps}_1= \pi(m+1/2)/\om+\eps(\eta_1/\om-C_{\Xi}-\Re e\,\Gamma_1-\pi^{-1}|\ln\eps|)\\[2pt]
\ell^{\eps}_2= \pi(m+1/2)/\om+\eps(\eta_2/\om-C_{\Xi}-\Re e\,\Gamma_2-\pi^{-1}|\ln\eps|)
\end{array}
\]
for some $m\in\N$ and $\eta_1\in\R$, $\eta_2\in\R$. Then when $\eps$ tends to zero, we have the following expansions
\begin{equation}\label{AsymptoFinalResultsTwo}
\fbox{$\begin{array}{l}
u^{\eps}(x,y)=W^+(x,y)+a_1\om\gamma_1(x,y)+a_2\om\gamma_2(x,y)+o(1) \quad\mbox{ in }\Om,\\[6pt]
u^{\eps}(x,y)=\eps^{-1}a_j\sin(\om (y-1))+O(1) \quad\mbox{ in }L^\eps_j,\ j=1,2,\\[6pt]
R_{+}^{\eps}=R_{+}^{0}+o(1),\qquad T^{\eps}=T^0+o(1), \\[6pt]
\mbox{where $a_1$, $a_2$, $R_{+}^{0}$, $T^0$ solve the system  (\ref{system1Two}).}
\end{array}$}
\end{equation}
Here $\gamma_1$, $\gamma_2$ are the functions introduced in (\ref{DefGammaMi}). 
\end{proposition}
\noindent Now, let us consider in more detail the particular case where $\Om=S=\R\times(0;1)$, \textit{i.e.} the case without obstacle. Then there holds $W^{\pm}=\mrm{w}^{\pm}$ and so $R_+=0$, $T=1$. Without loss of generality, impose $p_1=0$, so that $A_1=(0,1)$. In this situation, the system (\ref{system1Two}) simply writes
\begin{equation}\label{system1TwoSimple}
\begin{array}{|l}
R^{0}_+=i(a_1+a_2e^{i\om p_2})/2\\[4pt]
T^{0}=1+i(a_1+a_2e^{-i\om p_2})/2\\[4pt]
a_1(\eta_1+i/2)+a_2\om\tilde{\Gamma} =-1\\[3pt]
a_2(\eta_2+i/2)+a_1\om\tilde{\Gamma} =-e^{i\om p_2}.
\end{array}
\end{equation}
Pick $p_2>0$ large enough to neglect the evanescent part in $\gamma_1(A_2)=\tilde{\Gamma}$. From (\ref{relationCoefSca}), we obtain $\om\tilde{\Gamma}\approx ie^{i\om(p_2-p_1)}/2$. The last two lines of (\ref{system1TwoSimple}) give 
\begin{equation}\label{Expression apm}
a_1\approx\cfrac{ie^{2i\om p_2}/2-(\eta_2+i/2)}{e^{2i\om p_2}/4+(\eta_1+i/2)(\eta_2+i/2)}\,,\qquad a_2\approx\cfrac{-e^{i\om p_2}\eta_1}{e^{2i\om p_2}/4+(\eta_1+i/2)(\eta_2+i/2)}\,.
\end{equation}
In order to obtain $R^0_+\approx 0$ (remember that we are interested in constructing a phase shifter), from the first line of (\ref{system1TwoSimple}), we see that we must have $a_2/a_1=-e^{-i\om p_2}$ (we discard the case $a_2=a_1=0$ which directly gives $T^0=1$ which is not what we wish because we want to be able to control the phase). Relations (\ref{Expression apm}) imply
\[
\cfrac{a_2}{a_1}\approx \cfrac{-e^{i\om p_2}\eta_1}{ie^{2i\om p_2}/2-(\eta_2+i/2)}=-e^{-i\om p_2}\cfrac{\eta_1}{i/2-e^{-2i\om p_2}(\eta_2+i/2)}\,.
\]
As a consequence, we must set $\om p_2$ and $\eta_2$ such that $\eta_1=i/2-e^{-2i\om p_2}(\eta_2+i/2)$. This can be rewritten as follows:
\begin{equation}\label{SystemZeroRef}
e^{i\om p_2}\eta_1+e^{-i\om p_2}\eta_2+\sin(\om p_2)=0\quad\Leftrightarrow \quad\begin{array}{|l}
\cos(\om p_2)(\eta_1+\eta_2)+\sin(\om p_2)=0\\[2pt]
(\eta_1-\eta_2)\sin(\om p_2)=0.
\end{array}
\end{equation}
At this stage, we need to distinguish according to the case.\\
\newline 
$\star$ When $\om p_2\equiv 0\ [\pi]$, to get almost zero reflection, we must set the lengths of the resonators so that $\eta_1=-\eta_2$. In this situation, there holds $a_2=-a_1$ and so $T^0\approx1$ according to (\ref{AsymptoFinalResultsTwo}). Again, this is not interesting for our objective. \\ 
\newline
$\star$ When $\om p_2\not\equiv 0\ [\pi]$, to get almost zero reflection, according to (\ref{SystemZeroRef}), we see that we must set the lengths of the resonators so that $\eta_1=\eta_2$. Then inserting the relation $\eta_1-i/2=-e^{-2i\om p_2}(\eta_2+i/2)\Leftrightarrow\eta_2+i/2=-e^{2i\om p_2}(\eta_1-i/2)$ in (\ref{Expression apm}) leads to
\[
a_2\approx\cfrac{-e^{i\om p_2}\eta_1}{e^{2i\om p_2}/4-e^{2i\om p_2}(\eta_1^2+1/4)}=\cfrac{e^{-i\om p_2}}{\eta_1}\,.
\]
This yields 
\[
\begin{array}{rclcl}
T^0&=&1+\frac{i}{2}a_2( e^{-i\om p_2}-e^{i\om p_2})\\
&=&1+a_2\sin(\om p_2)=1-2\cos(\om p_2) e^{-i\om p_2}&=&1-2\cos^2(\om p_2)+i\sin(\om p_2) \\
&& &=&-\cos(2\om p_2)+i\sin(2\om p_2)\ =\ e^{i(\pi-2\om p_2)}\,.
\end{array}
\]
Thus we see that we can get $T^{\eps}\approx e^{i\mu}$ for any $\mu\in(0;2\pi)\setminus\{\pi\}$. To proceed, it suffices to take $p_2$ such that $\pi-2\om p_2\equiv\mu$ $[2\pi]$, \textit{i.e.} $p_2\equiv (\pi-\mu)/(2\om)$ $[\pi/\om]$. We are not surprise to find back the result of \S\ref{paragraphSuccessive} with the same distance between the resonators (see Procedure \ref{ProcePhaseShift1}). However there is one additional information that we get with this approach. Indeed, when there is no obstacle in $\Om$, we observe that the functions $\gamma_1$, $\gamma_2$ in (\ref{DefGammai}) are such that $\gamma_2(x,y)=\gamma_1(p_1+(x-p_2),y)$. As a consequence, we have $\Gamma_1=\Gamma_2$ in (\ref{DefGammaMi}). We deduce from (\ref{defEtaTwo}) that the resonators must have the same lengths to get almost zero reflection. This allows us to obtain the following procedure to construct the phase shifter. 
\begin{procedure}\label{ProcePhaseShift2}
Assume that we want to get a phase shifter with a transmission coefficient approximately equal to $e^{i\mu}$ for some $\mu\in(0;2\pi)\setminus\{\pi\}$.\\
1) Place a thin resonator at $A_1=(0,1)$ and a second one at $A_2=(p_2,1)$ with $p_2$ large enough such that $p_2\equiv (\pi-\mu)/(2\om)$ $[\pi/\om]$.\\
2) Then impose the same length to the two resonators and vary it to get almost zero reflection. According to the above analysis, this yields a transmission coefficient approximately equal to $e^{i\mu}$. 
\end{procedure}
\noindent In Figure \ref{FigPhase}, we use Procedure \ref{ProcePhaseShift2} to construct a geometry where we have zero reflection and a phase shift approximately equal to $\pi/4$. The method works correctly. In Figure \ref{FigPhaseCompa}, we compare the curves $(\ell_1,\ell_2)\mapsto|R^{\eps}_+(\ell_1,\ell_2)|$, $(\ell_1,\ell_2)\mapsto|T^{\eps}(\ell_1,\ell_2)|$ with $(\eta_1,\eta_2)\mapsto|R^{0}_+(\eta_1,\eta_2)|$, $(\eta_1,\eta_2)\mapsto|T^{0}(\eta_1,\eta_2)|$. Here 
$R_{+}^{0}$, $T^0$ solve the system  (\ref{system1Two}) and correspond to the main terms in the asymptotic of $R_{+}^{\eps}$, $T^\eps$. In accordance with Proposition \ref{PropoAsymTwoLig}, we observe a very good agreement between  the behaviours.

\begin{figure}[!ht]
\centering
\includegraphics[width=0.8\textwidth]{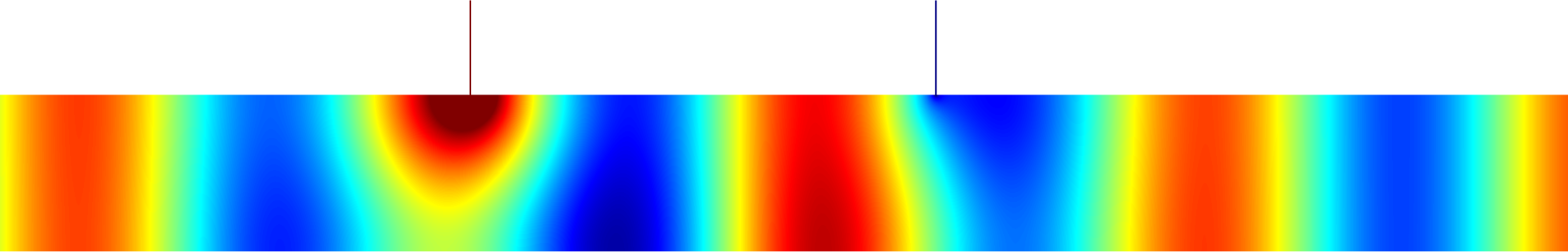}\\[5pt]
\includegraphics[width=0.8\textwidth]{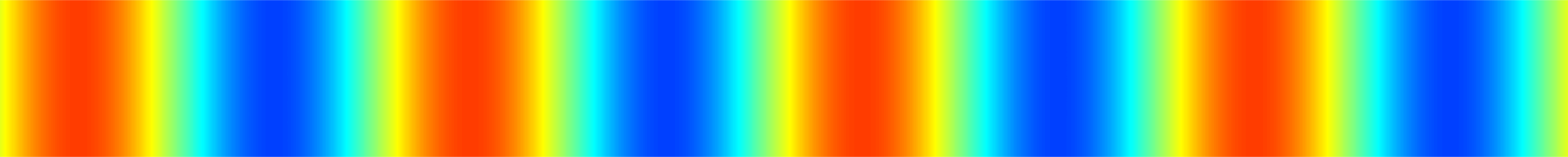}\\[5pt]
\includegraphics[width=0.8\textwidth]{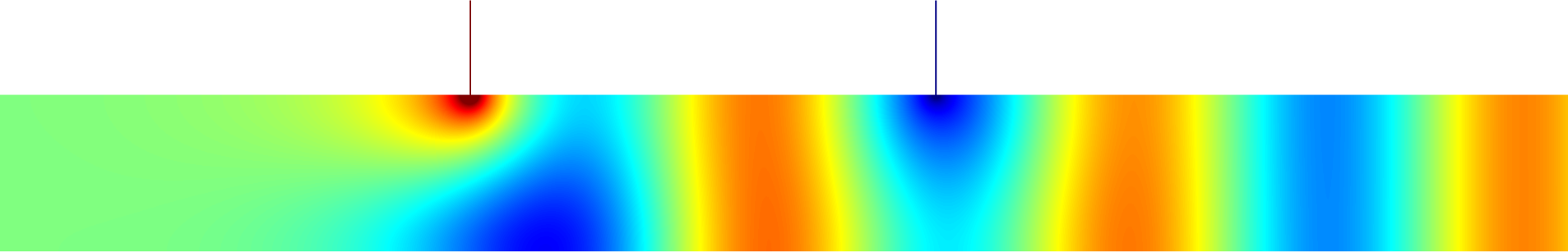}
\caption{Real parts of $u^\eps$ (top), $\mrm{w}^+$ (middle) and $u^\eps-\mrm{w}^+$ (bottom) ($\eps=0.01$). The resonators are tuned to get almost zero reflection and a phase shift approximately equal to $\pi/4$ (Procedure \ref{ProcePhaseShift2}).  \label{FigPhase}}
\end{figure}

\begin{figure}[!ht]
\centering
\includegraphics[trim={0.5cm 0.5cm 0cm 0cm},clip,width=0.32\textwidth]{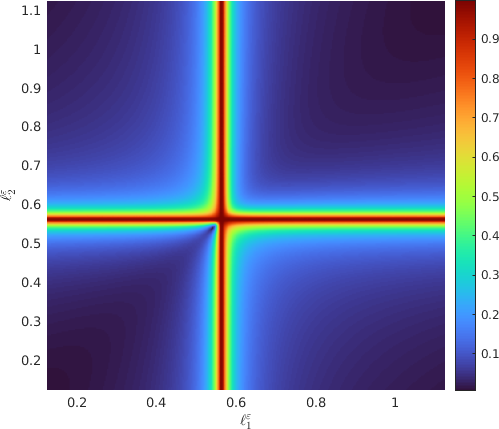}\qquad \includegraphics[trim={0.5cm 0.5cm 0cm 0cm},clip,width=0.32\textwidth]{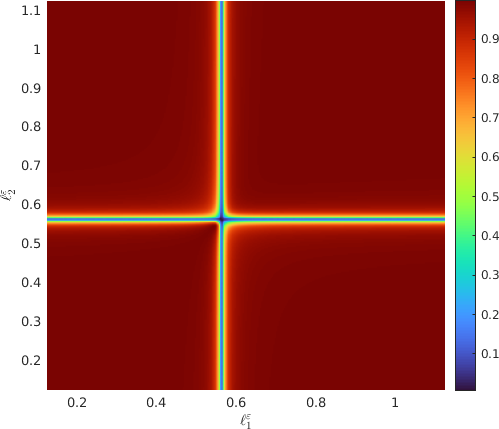}\\[5pt]
\includegraphics[trim={0.5cm 0.5cm 0cm 0cm},clip,width=0.32\textwidth]{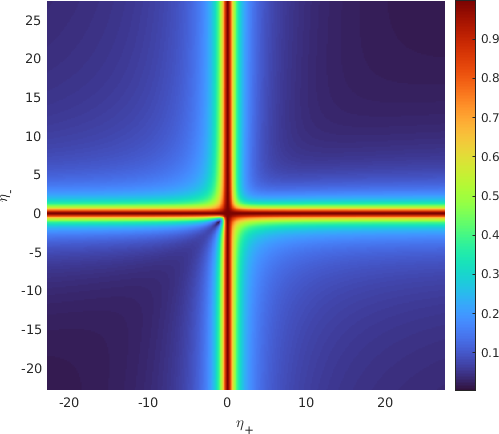}\qquad \includegraphics[trim={0.5cm 0.5cm 0cm 0cm},clip,width=0.32\textwidth]{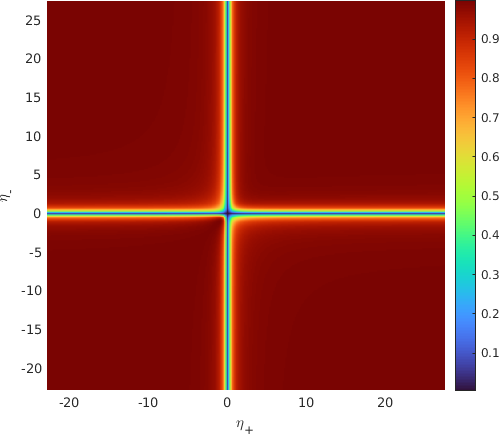}\\[-5pt]
\caption{First line: $(\ell_1,\ell_2)\mapsto|R^{\eps}_+(\ell_1,\ell_2)|$ (left) and  $(\ell_1,\ell_2)\mapsto|T^{\eps}(\ell_1,\ell_2)|$ (right).
Second line: $(\eta_1,\eta_2)\mapsto|R^{0}_+(\eta_1,\eta_2)|$ (left) and  $(\eta_1,\eta_2)\mapsto|T^{0}(\eta_1,\eta_2)|$ (right). We work in the geometry of Figure \ref{FigPhase} with $\eps=0.01$.
\label{FigPhaseCompa}}
\end{figure}

\newpage

\section{Cloaking}\label{SectionCloaking}

\subsection{Cloaking with three resonators}

Gathering the results of Sections \ref{SectionZeroR} and \ref{SectionPhaseShift}, now we can propose a method to (approximately) cloak any obstacle.

\begin{procedure}\label{ProceCloakingThree}
Let $\Om$ be a given waveguide as described before (\ref{MainPb}). \\
1) Following Procedure \ref{MainProce}, add a thin resonator to $\Om$ and tune its length to get almost zero reflection.\\
2) Measure the phase of the transmission coefficient in the geometry obtained after step 1).\\ 
3) Following Procedure \ref{ProcePhaseShift1} or Procedure \ref{ProcePhaseShift2}, place additionally two other resonators and tune their lengths to compensate the phase shift. This yields a geometry where the transmission coefficient is approximately equal to one.
\end{procedure}

\noindent In Figure \ref{FigCloakDino}, we use Procedure \ref{ProceCloakingThree} to approximately cloak an obstacle. The scattered field $u^\eps-\mrm{w}^+$ is indeed approximately exponentially decaying at infinity. We observe that one of the three resonators has a very weak influence. This is a particular circumstance for this geometry which is related to the analysis presented in the next paragraph.

\begin{figure}[!ht]
\centering
\includegraphics[width=0.92\textwidth]{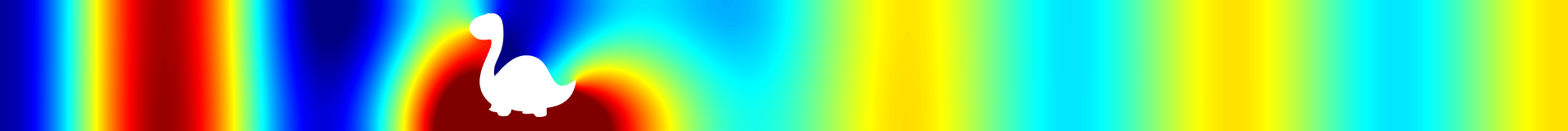}\\[5pt]
\includegraphics[width=0.92\textwidth]{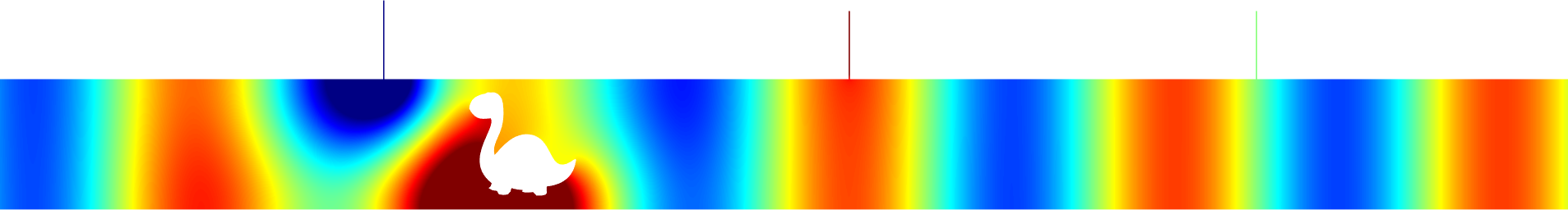}\\[5pt]
\includegraphics[width=0.92\textwidth]{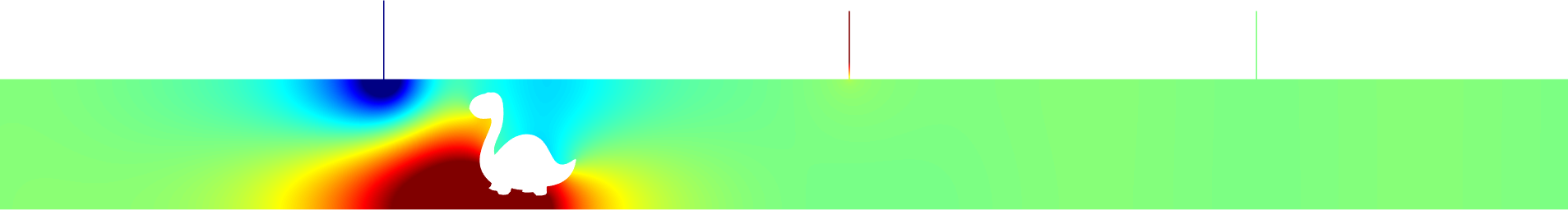}
\caption{Real parts of $W^+$ (top), $u^\eps$ (middle) and $u^\eps-\mrm{w}^+$ (bottom). The resonators are tuned to get approximately $T^{\eps}$ equal to one (Procedure \ref{ProceCloakingThree}). Here $\eps=0.01$.\label{FigCloakDino}}
\end{figure}

\subsection{Cloaking with less than three resonators}

In the previous section, we explained how to cloak any object with three resonators. Now we study if we can diminish the number of resonators needed in the process. In Section \ref{SectionZeroR}, we described how to cancel approximately the reflection produced by the obstacle by working with one resonator. In the following proposition, we show that in certain circumstances, the corresponding transmission coefficient can be approximately equal to one. 
\begin{proposition}\label{PropositionCloakingOneLig}
Assume that $\Om$ is such that $T\in\mathscr{C}(1/2,1/2)\setminus\{0,1\}$ ($T$ is defined in (\ref{Field1})). Then there are some positions of $A=(p,1)$ such that there holds $T^{0}(\eta)=1$ for some $\eta\in\R$. Here $T^{0}$ is the main term in the asymptotics of $T^{\eps}$, see (\ref{AsymptoFinalResults1}). For such $A$, Proposition \ref{PropoAsymOneLig} ensures that one can make $T^{\eps}$ as close as one wishes to $1$ by taking $\eps$ small enough and by tuning correctly the length of the resonator $L^{\eps}$.
\end{proposition}
\begin{proof}
According to Proposition \ref{PropositionZeroRef}, we can find $A$ and $\eta_{\star}$ such that $R^{0}_+(\eta_{\star})=0$ and $W^+(A)W^-(A)\ne0$. In this case relation (\ref{RelConsNRJ}) guarantees that $|T^{0}(\eta_{\star})|=1$. We also know that the set $\{T^{0}(\eta)\,|\,\eta\in\overline{\R}:=\R\cup{\pm\infty}\}$ coincides with the circle passing through $0$ (see the analysis of \S\ref{ParagZeroT}), $T$ and $T^{0}(\eta_{\star})$. Observe that they are exactly two circles passing through $0$, $T$ and hitting only once the unit circle (we impose this condition due to conservation of energy (\ref{RelConsNRJ})), one of them being $\mathscr{C}(1/2,1/2)$. Our goal is to show that we can choose $A$ such that $\{T^{0}(\eta)\,|\,\eta\in\overline{\R}\}=\mathscr{C}(1/2,1/2)$.\\
\newline
According to (\ref{SolutionSystem}), since $R^{0}_+(\eta_{\star})=0$, we have 
\[
T^{0}(\eta_{\star})=T-\cfrac{W^{-}(A)}{W^{+}(A)}\,R_+.
\]
Set again $R_+=\rho e^{i\theta_+}$ with $\rho\in(0;1)$ and $\theta_+\in [0;2\pi)$. Using that $T\in\mathscr{C}(1/2,1/2)\setminus\{0,1\}$ with $|T|=\sqrt{1-\rho^2}$ (due to conservation of energy), a direct calculus gives 
\begin{equation}\label{DefParamTau}
T=1-\rho^2+ i\tau\,\rho\sqrt{1-\rho^2},
\end{equation}
with $\tau=1$ or $\tau=-1$. On the other hand, if the resonator is quite far on the left of the obstacle, we have $W^+(A)\approx e^{i\om p}+R_+\,e^{-i\om p}$ as well as $W^-(A)\approx T\,e^{-i\om p}$, which implies 
\begin{equation}\label{EqAudessus}
T^{0}(\eta_{\star})\approx \cfrac{e^{2i\om p}\,T}{e^{2i\om p}+R_+}=\cfrac{T}{1+R_+\,e^{-2i\om p}}\,.
\end{equation}
According to Proposition \ref{PropositionZeroRef} (see in particular (\ref{RelationToGetZeroRLeft})), $A$ can be chosen such that $\cos(2\om p-\theta_+)\approx-\rho$. Then we have $\sin(2\om p-\theta_+)\approx \pm\sqrt{1-\rho^2}$. By choosing $A$ such that additionally $\sin(2\om p-\theta_+)\approx \tau\,\sqrt{1-\rho^2}$ where $\tau$ is defined above, which is doable, from (\ref{EqAudessus}) we get $T^{0}(\eta_{\star})\approx 1$. This is enough to guarantee that when $p$ is sufficiently small, with the above conditions, we have $\{T^{0}(\eta)\,|\,\eta\in\overline{\R}\}=\mathscr{C}(1/2,1/2)$ and so $T^{0}(\eta_{\star})= 1$. 
\end{proof}
\noindent This proposition allows us to derive the following procedure.
\begin{procedure}\label{ProceCloakingOne}
Let $\Om$ be a given waveguide as described before (\ref{MainPb}). Assume that $R_\pm=\rho e^{i\theta_\pm}$ with $\rho\in(0;1)$, $\theta_\pm\in [0;2\pi)$ and that $T\in\mathscr{C}(1/2,1/2)\setminus\{0,1\}$.\\
1) Measure $\rho$, $\theta_\pm$ and $T$.\\ 
2) Place one resonator quite far on the left (resp. on the right) of the obstacle at the position $A=(p,1)$ with $p$ such that $\cos(2\om p-\theta_+)=-\rho$ and 
$\sin(2\om p-\theta_+)= \tau\sqrt{1-\rho^2}$ (resp. $\cos(2\om p+\theta_-)=-\rho$ and 
$\sin(2\om p+\theta_-)= \tau\sqrt{1-\rho^2}$) where $\tau$ appears is (\ref{DefParamTau}). Tune its length to get zero reflection. Then the transmission coefficient will be approximately equal to one.
\end{procedure}

\noindent For a general waveguide $\Om$ however, we do not have $T\in\mathscr{C}(1/2,1/2)\setminus\{0,1\}$. Can we perturb the geometry by adding one resonator to get this property?

\begin{proposition}\label{PropositionCloakingOneLig}
Assume that $\Om$ is such that $T\notin\mathscr{C}(1/2,1/2)\setminus\{0,1\}$ ($T$ is defined in (\ref{Field1})). Then there are some positions of $A=(p,1)$ such that the circle $\{T^{0}(\eta)\,|\,\eta\in\overline{\R}\}$ has a non empty intersection with $\mathscr{C}(1/2,1/2)\setminus\{0,1\}$.  Here $T^{0}$ is the main term in the asymptotics of $T^{\eps}$, see (\ref{AsymptoFinalResults1}). For such $A$, Proposition \ref{PropoAsymOneLig} ensures that we can have $T^{\eps}$ on $\mathscr{C}(1/2,1/2)\setminus\{0,1\}$ by taking $\eps$ small enough and by tuning correctly the length of the resonator $L^{\eps}$.
\end{proposition}
\begin{proof}
The analysis of \S\ref{ParagZeroT} ensures that $\mathcal{C}:=\{T^{0}(\eta)\,|\,\eta\in\overline{\R}\}$ is a circle passing through zero and $T$. Moreover, due to conservation of energy (see relation (\ref{RelConsNRJ})), $\mathcal{C}$ must be inside the unit disk. Therefore to establish that $\mathcal{C}\cap \{\mathscr{C}(1/2,1/2)\setminus\{0,1\}\}\ne\emptyset$, it is sufficient to show that the center $\mathscr{O}_\mathcal{C}$ of $\mathcal{C}$ has a non zero imaginary part. Using (\ref{SolutionSystem}), (\ref{Mobius}), we find that 
\[
\mathscr{O}_\mathcal{C}=T-\cfrac{W^{+}(A)W^{-}(A)}{|W^+(A)|^2+|W^-(A)|^2}\,.
\]
If the resonator is located quite far on the left of the obstacle, we have $W^+(A)\approx e^{i\om p}+R_+\,e^{-i\om p}$ as well as $W^-(A)\approx T\,e^{-i\om p}$. This yields
\begin{equation}\label{PhaseVanishes}
\mathscr{O}_c\approx T-\cfrac{T(1+R_+\,e^{-2i\om p})}{2+2\Re e\,(R_+\,e^{-2i\om p})}=\cfrac{T}{2}\,\left(1-\cfrac{i\,\Im m\,(R_+\,e^{-2i\om p})}{1+\Re e\,(R_+\,e^{-2i\om p})}\right).
\end{equation}
The imaginary part of the right hand side of (\ref{PhaseVanishes}) can vanish for a set of values of $p$  which is at most discrete. This is enough to conclude to the existence of some $A=(p,1)$ such that $\Im m\,\mathscr{O}_c\ne0$.
\end{proof}

\noindent This yields the following strategy to do cloaking with two resonators. 
\begin{procedure}\label{ProceCloakingTwo}
Let $\Om$ be a given waveguide as described before (\ref{MainPb}). Assume that $R_\pm=\rho e^{i\theta_\pm}$ with $\rho\in(0;1)$, $\theta_\pm\in [0;2\pi)$ and that $T\notin\mathscr{C}(1/2,1/2)$.\\
1) Add one resonator at the position $A=(p,1)$ and tune its length so that the transmission coefficient in the new geometry belongs to $\mathscr{C}(1/2,1/2)\setminus\{0,1\}$. This is possible according to Proposition \ref{PropositionCloakingOneLig}.\\ 
2) Then add a second resonator following Procedure \ref{ProceCloakingOne} to get a transmission coefficient approximately equal to one.
\end{procedure}

\noindent In Figures \ref{FigCloakElephant}--\ref{FigCloakPenetrable}, we use Procedure \ref{ProceCloakingTwo} to approximately cloak different obstacles. In each of the situations, we indeed observe that the scattered field $u^\eps-\mrm{w}^+$ is approximately exponentially decaying at infinity. For Figures \ref{FigCloakElephant}--\ref{FigCloakMario}, the geometry is described explicitly by the pictures. Let us mention that the configuration of Figure \ref{FigCloakTuyau} is a difficult one because the initial transmission coefficient is very close to zero (see the top picture). In this situation, the length of the resonators must be tuned very precisely. One possible option to facilitate the approach, as explained in Remark \ref{RmkIntermLig}, is to work with one or several intermediate resonators to get a transmission coefficient with a larger modulus and then to apply Procedure \ref{ProceCloakingTwo}. Finally in Figure \ref{FigCloakPenetrable}, we consider the case of a penetrable obstacle coinciding with the fish of Figures \ref{FigFish01}, \ref{FigFish001}. More precisely, we replace the first equation of (\ref{MainPb}) by $\Delta u+\om^2nu=0$ in $\Om$ with an index material $n$ equal to $1$ outside of the inclusion and equal to $6$ inside. This does not enter strictly the framework introduced in Section \ref{SectionSetting} but can be dealt with in a completely similar way.

\begin{figure}[!ht]
\centering
\includegraphics[width=0.92\textwidth]{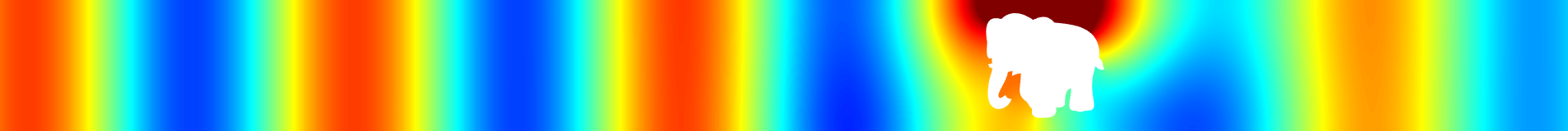}\\[5pt]
\includegraphics[width=0.92\textwidth]{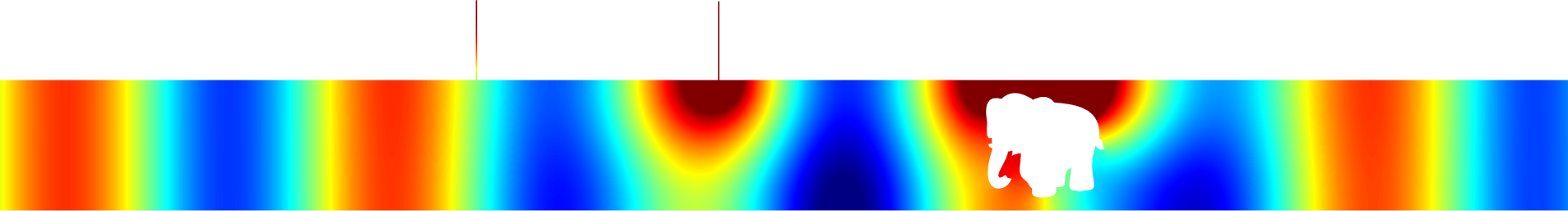}\\[5pt]
\includegraphics[width=0.92\textwidth]{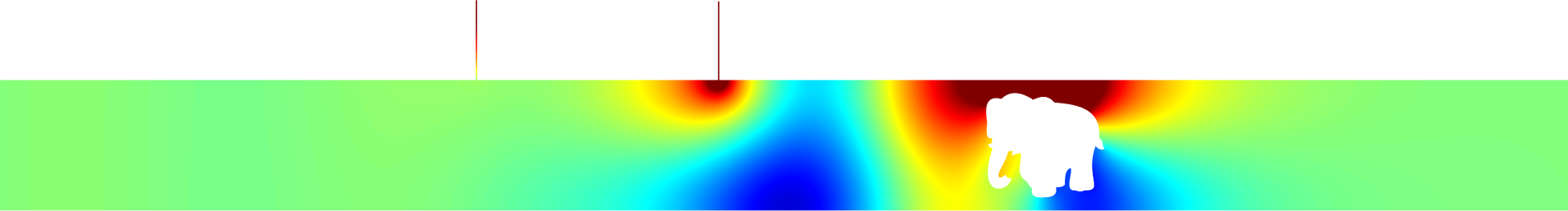}
\caption{Real parts of $W^+$ (top), $u^\eps$ (middle) and $u^\eps-\mrm{w}^+$ (bottom). The resonators are tuned to get approximately $T^{\eps}$ equal to one (Procedure \ref{ProceCloakingTwo}). Here $\eps=0.01$.\label{FigCloakElephant}}
\end{figure}

\begin{figure}[!ht]
\centering
\includegraphics[trim={5cm 0cm 6cm 0cm},clip,width=0.92\textwidth]{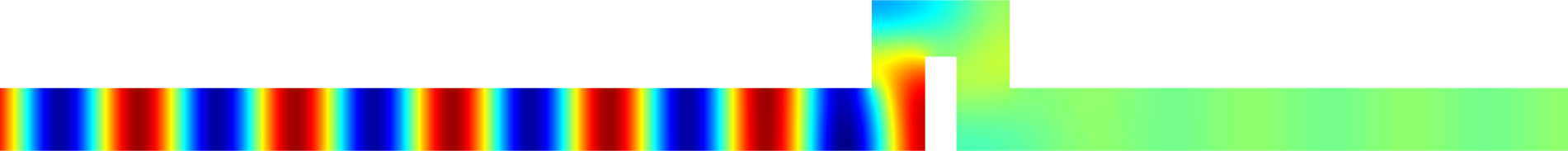}\\[5pt]
\includegraphics[trim={5cm 0cm 6cm 0cm},clip,width=0.92\textwidth]{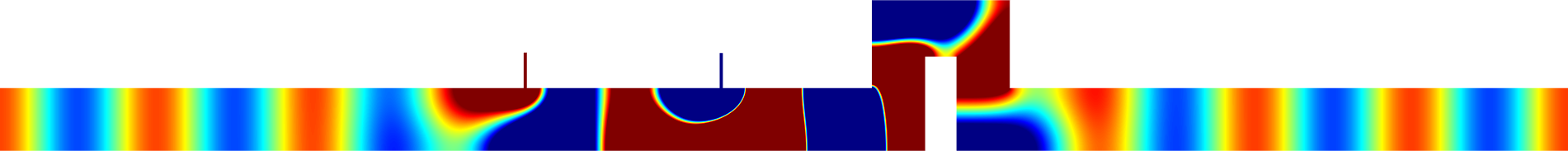}\\[5pt]
\includegraphics[trim={5cm 0cm 6cm 0cm},clip,width=0.92\textwidth]{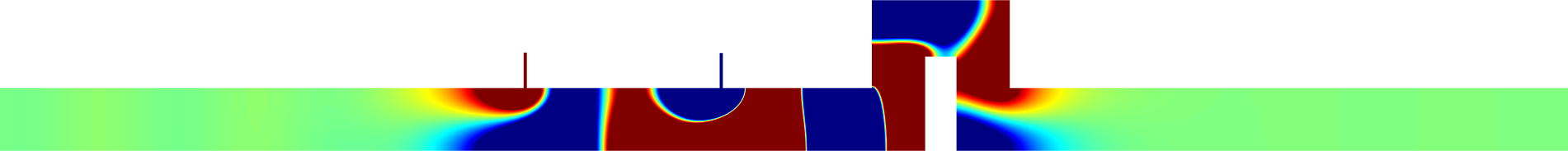}
\caption{Real parts of $W^+$ (top), $u^\eps$ (middle) and $u^\eps-\mrm{w}^+$ (bottom). The resonators are tuned to get approximately $T^{\eps}$ equal to one (Procedure \ref{ProceCloakingTwo}). Here $\eps=0.05$.\label{FigCloakTuyau}}
\end{figure}

\begin{figure}[!ht]
\centering
\includegraphics[width=0.92\textwidth]{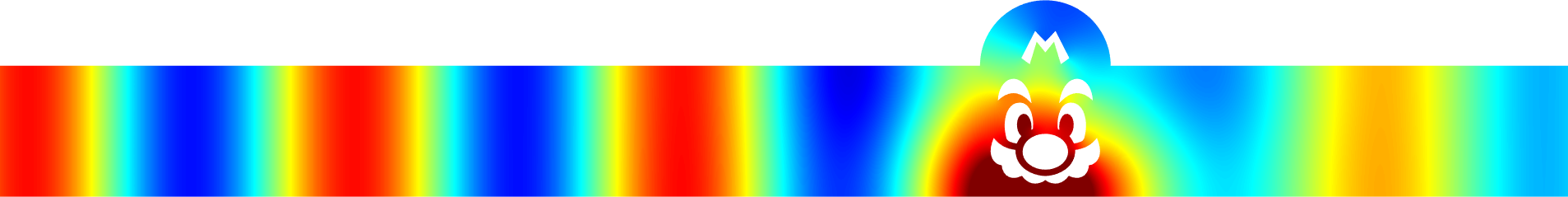}\\[5pt]
\includegraphics[width=0.92\textwidth]{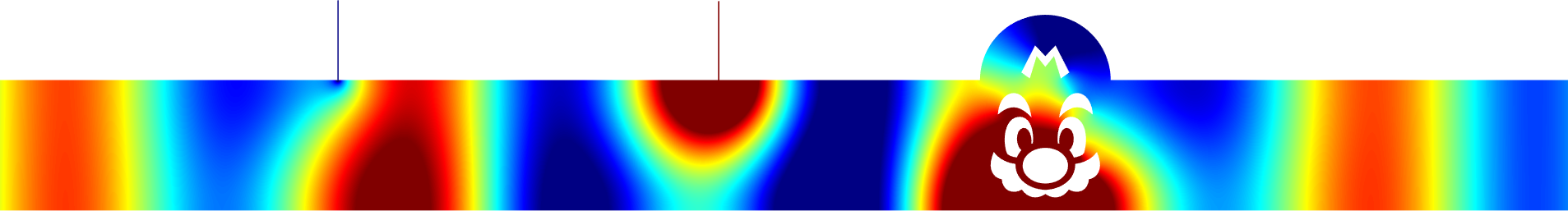}\\[5pt]
\includegraphics[width=0.92\textwidth]{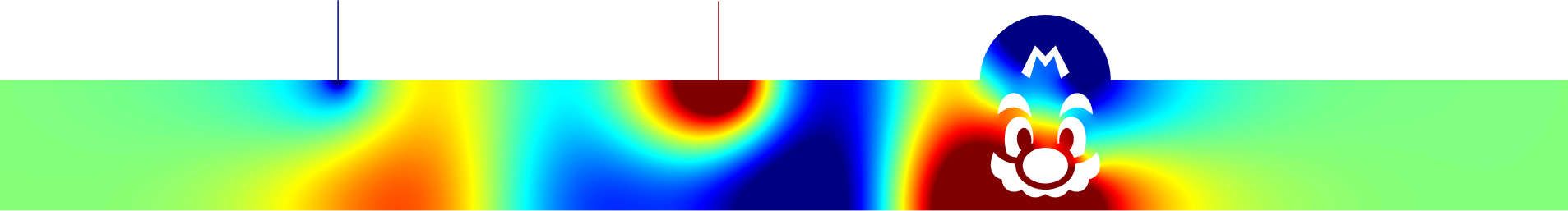}
\caption{Real parts of $W^+$ (top), $u^\eps$ (middle) and $u^\eps-\mrm{w}^+$ (bottom). The resonators are tuned to get approximately $T^{\eps}$ equal to one (Procedure \ref{ProceCloakingTwo}). Here $\eps=0.01$.\label{FigCloakMario}}
\end{figure}

\begin{figure}[!ht]
\centering
\includegraphics[width=0.92\textwidth]{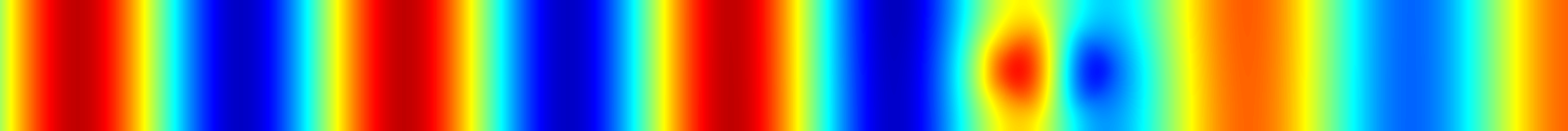}\\[5pt]
\includegraphics[width=0.92\textwidth]{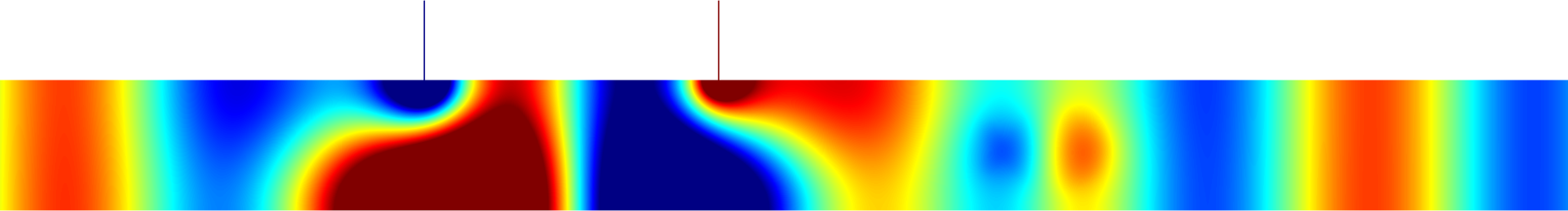}\\[5pt]
\includegraphics[width=0.92\textwidth]{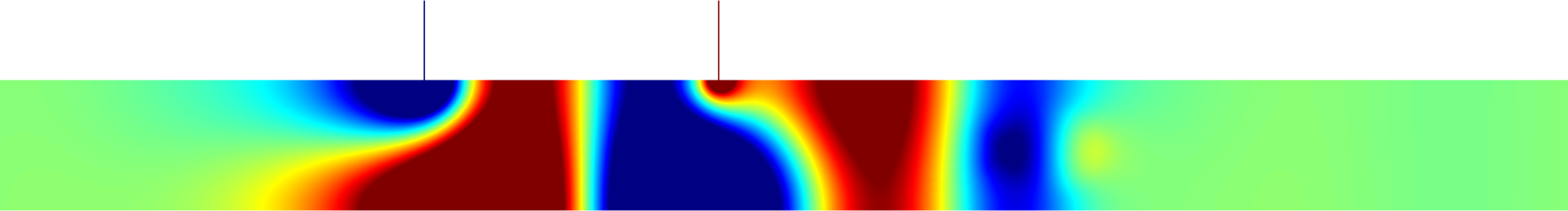}
\caption{Real parts of $W^+$ (top), $u^\eps$ (middle) and $u^\eps-\mrm{w}^+$ (bottom). The resonators are tuned to get approximately $T^{\eps}$ equal to one (Procedure \ref{ProceCloakingTwo}). The obstacle has the same shape as in Figures \ref{FigFish01}, \ref{FigFish001} but here it is penetrable. The index material is set to $n=6$ and $\eps=0.01$.\label{FigCloakPenetrable}}
\end{figure}

\newpage ~\\[-40pt]

\section{Concluding remarks}\label{SectionConclusion}

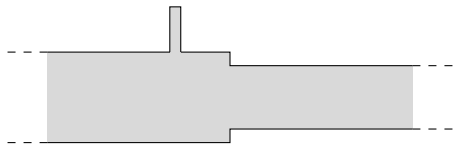
\begin{figure}[!ht]
\centering
\begin{tikzpicture}[scale=1.2]
\draw[fill=gray!30,draw=none](-2,0) rectangle (0,1);
\draw[fill=gray!30,draw=none](0,0.15) rectangle (2,0.85);
\draw (-2,0)--(0,0)--(0,0.15)--(2,0.15);
\draw (-2,1)--(0,1)--(0,0.85)--(2,0.85);
\draw[fill=gray!30,draw=none](-0.66,0.9) rectangle (-0.54,1.5);
\draw (-0.66,1)--(-0.66,1.5)--(-0.54,1.5)-- (-0.54,1);
\draw[dashed] (-2,1)--(-2.5,1);
\draw[dashed] (-2,0)--(-2.5,0);
\draw[dashed] (2,0.15)--(2.5,0.15);
\draw[dashed] (2,0.85)--(2.5,0.85);
\end{tikzpicture}
\caption{A device where the conclusions drawn in this article do not apply directly. \label{InoperantDomain}} 
\end{figure}

\noindent $i)$ We considered straight vertical thin resonators to simplify the presentation and to limit the complexity of the notation. We could have worked similarly  with resonators coinciding at the limit $\eps\to0$ with some smooth curves, we would have observed the same phenomena. What matters is the lengths of the thin resonators. Additionally, the orientation of the resonators (if they are not vertical) plays no major role at first order in the asymptotic (see more details in \cite{ChNaTa20}).\\ 
\newline
$ii)$ What was done in 2D here could be adapted in 3D. The asymptotic expansion would be different, in particular because the equivalent of the function $Y^1$ introduced in (\ref{PolyGrowth}) has a different behaviour at infinity in 3D, but the methodology would be the same (see \cite{NaCh21a,NaCh21b} for related works). \\
\newline
$iii)$ On the contrary, the approach proposed in this article is very specific to Neumann Boundary Conditions (BCs) and cannot be adapted for Dirichlet BCs. Indeed, with Dirichlet BCs nothing enters into the thin resonators and almost all of the incoming energy is backscattered. Therefore, to do cloaking with Dirichlet BCs, it is necessary to find a different idea.\\
\newline
$iv)$ In Figure \ref{FigZeroT}, we have seen that we can work in geometries which, outside of the obstacle, do not necessarily coincide with the reference strip. In particular, bended waveguides can be considered. However the incoming branch and the outgoing branch must have the same width. The situation where this condition is not met (see Figure \ref{InoperantDomain} and for example the articles \cite{CaNR12,CaNR13}) can be studied with a completely similar approach but the results will be different. For example, nothing guarantee that almost zero reflection can be obtained with a single resonator.\\
\newline
$v)$ The question of working at higher wavenumber, so that several modes can propagate, remains largely open. In this case, to cloak the obstacle, it is necessary to control more than two scattering coefficients. A natural idea is to increase the number of thin resonators to get more means of action. In this situation, the asymptotic analysis can be adapted quite directly from what has been done in this work. But then we  are led to work on an algebraic system which seems at first glance hard to handle (extrapolate from (\ref{system1Two})). In particular, coupling effects do not facilitate the analysis.\\
\newline
$vi)$ The methodology proposed in this article allows us to do approximate cloaking, that is to achieve $T^{\eps}=1+o(\eps)$ as $\eps$ tends to zero. A natural question, at least from a theoretical point of view, is ``can we get $T^{\eps}=1$ exactly?''. This may be doable by adapting some arguments presented in \cite{BoCN18}. Let us try to summarize the idea. By adding three non resonant thin chimneys, as shown in \cite{BoCN18}, we could be able to cancel the term at order $\eps$ to obtain $T^{\eps}=1+o(\eps^2)$. To get $T^{\eps}=1$, we need to kill all the terms in the asymptotic and this is not simple because at order $p\ge2$, the dependence on the geometrical parameters of the thin chimneys becomes non linear. But it was shown in \cite{BoCN18} (see also \cite{na546} for a similar approach) that one can perturb the length of the thin chimneys, with a perturbation determined by solving a fixed point problem, which permits to tackle the problem of the non linearity, to achieve $T^{\eps}=1$. The context of the present article however is a bit different from \cite{BoCN18} and the complete justification of the scheme must be studied carefully. Results of \cite{BeBC21} to prove the invertibility of some differentials involved in the process may be useful.\\
\newline
$vii)$ The question of mimicking (see \cite{BeBC21} in the context of thermal cloaking) the scattering properties at infinity, that is to perturb the geometry so that the scattering coefficients become approximately equal to those of a given obstacle, may be studied with the tools presented in this work. \\
\newline
$viii)$ Assume that the thin resonators have been tuned to be resonant at a certain wavenumber $\om_1\in(0;\pi)$. Now if we work at $\om_2\ne\om_1$, in general the thin resonators will only perturb the field $u^{\eps}$, and so the scattering coefficients, by a term of order $\eps$. This is interesting because it allows us to decouple the action of the thin resonators at order $1$. As a consequence, by working with $2p$ thin resonators, we can approximately cloak any obstacle at the discrete collection of wavenumbers $\{\om_1,\om_2,\dots,\om_p\}\subset(0;\pi)^{p}$. Of course imposing invisibility for a continuum of wavenumbers is another question, which may be impossible to solve, see the related works \cite{SoGK07,Norr15,MoAl16,CaMi17,Norr18}.

\section*{Appendix: auxiliary results}

\textbf{Proof of Proposition \ref{propositionStructure}}. We reproduce the material of \cite[Proposition 3.4]{BeBC21}. Looking at the behaviour of $\mathbb{S}\overline{W}-W$ for $|x|>d$ and using that $\mathbb{S}$ is unitary (see relations (\ref{RelConsNRJ})), one finds that $\mathbb{S}\overline{W}-W$ is a vector of functions which solve the homogeneous problem (\ref{MainPb}) and which are exponentially decaying at infinity. In other words, $\mathbb{S}\overline{W}-W$ is a vector of trapped modes. But since by definition the $W^{\pm}$ are orthogonal to trapped modes for the $\mL^2(\Om)$ inner product, we deduce that $\mathbb{S}\overline{W}=W$.

\begin{lemma}\label{LemmaCReal}
The constant $C_{\Xi}$ appearing in the decomposition (\ref{PolyGrowth}) of the function $Y^1$ is real.
\end{lemma}
\begin{proof}
Since there holds $\Delta Y^1=0$ in $\Xi$, for all $\kappa>0$, we have 
\[
0=\int_{\Xi^\kappa} (Y^1-\overline{Y^1})\Delta Y^1-\Delta(Y^1-\overline{Y^1})Y^1\,d\xi_xd\xi_y
\]
with $\Xi^\kappa:=\{(\xi_x,\xi_y)\in\Xi,\,\xi_y<0\mbox{ and }|\xi|<\kappa\}\cup \{(\xi_x,\xi_y)\in(-1/2;1/2)\times[0;\kappa)\}$. Integrating by parts and taking the limit $\kappa\to+\infty$, we get $C_\Xi-\overline{C_\Xi}=0$. This shows that $C_\Xi$ is real.
\end{proof}

\begin{lemma}\label{lemmaRelConstants}
The constant $\Gamma$ corresponding to the constant behaviour of $\gamma$ at $A$ (see (\ref{DefGammaM})) is such that 
\[
\Im m\,(\om\Gamma)=(|W^+(A)|^2+|W^-(A)|^2)/4. 
\]
\end{lemma}
\begin{proof}
Since the function $\gamma$ is outgoing, we have the expansion $\gamma=s_{\pm}\mrm{w}^\pm+\tilde{\gamma}$ for $\pm x >d$ where $s_{\pm}\in\Cplx$ and where $\tilde{\gamma}$ is exponentially decaying at infinity. Integrating by parts in
\[
0=\int_{\Om^\kappa}(\Delta \gamma +\omega^2\gamma)W^{\pm}-\gamma\,(\Delta W^{\pm} +\omega^2W^{\pm})\,dz,
\]
and taking the limit $\kappa\to+\infty$ as in (\ref{IntegRho1}), we obtain 
\begin{equation}\label{relationCoefSca}
s_{\pm}=iW^{\mp}(A)/(2\om).
\end{equation}
On the other hand, integrating by parts in
\[
0=\int_{\Om^\kappa}(\Delta \gamma +\omega^2\gamma)\overline{\gamma}-\gamma\,(\Delta \overline{\gamma} +\omega^2\overline{\gamma})\,dz,
\]
and taking again the limit $\kappa\to+\infty$, we obtain $2\om(|s_{-}|^2+|s_{+}|^2)-2\Im m\,\Gamma=0$. From (\ref{relationCoefSca}), this yields the desired result.
\end{proof}

\begin{lemma}\label{lemmaRelConstantsTwo}
Let $\gamma_j$, $j=1,2$, be the functions introduced in (\ref{DefGammai}). We have $\gamma_1(A_2)=\gamma_2(A_1)$. 
\end{lemma}
\begin{proof}
Integrating by parts in
\[
0=\int_{\Om^\kappa}(\Delta \gamma_1 +\omega^2\gamma_1)\gamma_2-\gamma_1\,(\Delta \gamma_2 +\omega^2\gamma_2)\,dz,
\]
and taking the limit $\kappa\to+\infty$, we find $\gamma_1(A_2)=\gamma_2(A_1)$. 
\end{proof}

\section*{Acknowledgments} 
The work of S.A. Nazarov was supported by the Ministry of Science and Higher Education of Russian Federation within the framework of the Russian State Assignment under contract No.  FFNF-2021-0006.
The authors wish to thank Lorenzo Scaglione who worked on parts of this article during its internship, in particular on the result leading to Proposition \ref{PropositionZeroRefSym}.

\bibliography{Bibli}
\bibliographystyle{plain}
\end{document}